\documentclass[12pt]{amsart}

\usepackage[hmargin=0.8in,height=8.6in]{geometry}
\usepackage{amssymb,amsthm, times}
\usepackage{delarray,verbatim}
\usepackage{bm}
\usepackage{ifpdf}
\ifpdf
\usepackage[pdftex]{graphicx}
 \else
\usepackage[dvips]{graphicx}
 \fi

\usepackage{ifpdf}
\usepackage{color}
\definecolor{webgreen}{rgb}{0,.5,0}
\definecolor{webbrown}{rgb}{.8,0,0}
\definecolor{emphcolor}{rgb}{0.95,0.95,0.95}

\usepackage{hyperref}
\hypersetup{%
          colorlinks=true,
          linkcolor=webbrown,
          filecolor=webbrown,
          citecolor=webgreen,
          breaklinks=true}
\ifpdf \hypersetup{pdftex,
            pdfstartview=FitH, 
            bookmarksopen=true,
            bookmarksnumbered=true
} \else \hypersetup{dvips} \fi

\numberwithin{equation}{section}

\newtheorem{proposition}{Proposition}[section]

\newtheorem{remark}{Remark}[section]
\newtheorem{lemma}{Lemma}[section]
\newtheorem{example}{Example}[section]

\numberwithin{remark}{section} \numberwithin{proposition}{section}
\numberwithin{corollary}{section}
\newcommand{\lapinv}{\Phi(q)}
\newcommand {\R}{\mathbb{R}}

\newcommand {\p}{\mathbb{P}}

\newcommand {\E}{\mathbb{E}}

\newcommand{\diff}{{\rm d}}

\newcommand{\lev}{L\'{e}vy }

\title[Phase-type fitting of scale functions for spectrally negative L\'{E}VY processes]{Phase-type fitting of scale functions for spectrally negative L\'{E}VY processes}

\thanks{Date: \today.  An earlier version of this paper was circulated as
``On Scale Functions of Spectrally Negative Levy Processes with Phase-Type Jumps".  M.\ Egami is in part supported
by Grant-in-Aid for Scientific Research (B) No.\ 23330104 and (B) No.\ 22330098. K.\
Yamazaki is in part supported by Grant-in-Aid for Young Scientists
(B) No.\ 22710143, the Ministry of Education, Culture, Sports,
Science and Technology, and by Grant-in-Aid for Scientific Research (B) No.\  2271014, Japan Society for the Promotion of Science. The authors thank the anonymous referees for their thorough reviews and insightful comments that help improve the presentation of this paper.}

\author[M. Egami]{\,Masahiko Egami$^\dag$\,}\thanks{$^\dag$\,Graduate School of Economics,
Kyoto University, Sakyo-Ku, Kyoto, 606-8501, Japan }

\author[K. Yamazaki]{\,Kazutoshi Yamazaki$^\ddag$}\thanks{$\ddag$\, (corresponding author) Department of Mathematics,
Faculty of Engineering Science, Kansai University, Suita-shi, Osaka 564-8680, Japan. Email: \mbox{{\em
kyamazak@kansai-u.ac.jp}}.  Tel: +81-6-6368-1527.}
\thanks{}

\begin{document}

\begin{abstract}
\noindent We study the scale function of the spectrally
negative phase-type \lev process.   Its scale function admits an analytical expression and so do a number of its fluctuation identities.  Motivated by the fact that the class of phase-type
distributions is dense in the class of all positive-valued
distributions, we propose a new approach to approximating the scale
function and the associated fluctuation identities for a general spectrally negative \lev process.  Numerical examples are provided to illustrate the effectiveness of the approximation method.  

\end{abstract}


\maketitle \noindent \small{\textbf{Key words:} 
phase-type models; spectrally negative \lev processes; scale functions  }\\

\section{Introduction} \label{section_introduction}

In the last decade, significant progress has been made regarding
spectrally negative \lev processes and their scale functions.  As can be seen in the work of, for example,
\cite{Bertoin_1996, Doney_2007, Kyprianou_2006},
a number of fluctuation identities concerning spectrally negative \lev
processes can be expressed in terms of scale functions. There are
numerous applications in applied probability including optimal
stopping, queuing, mathematical finance, insurance and credit risk.
Despite these advances, a
major obstacle still remains in putting these in practice because scale functions are in general known
only up to their Laplace transforms, and only a few cases
admit explicit expressions.  
Typically, one needs to rely on numerical Laplace inversion in order to approximate the scale function; see \cite{Kuznetsov_2011, Surya_2008}.

In this paper, we propose a \emph{phase-type (PH)-fitting} approach by using the scale function for the class of spectrally negative PH \lev processes, or \lev processes with negative PH-distributed jumps. Consider a continuous-time Markov chain with some initial distribution and state space consisting of a single absorbing state and a finite number of transient states.  The PH-distribution is the distribution of the time to absorption.
The class of PH-distributions includes, for example, the exponential, hyperexponential,  Erlang, hyper-Erlang and Coxian distributions; see, for example,\ Section 3 of \cite{Asmussen_2003}.

The scale function of the spectrally negative PH \lev process can be obtained by analytically applying Laplace inversion; more generally, whenever the Laplace exponent is a rational function, the scale function can be computed directly by partial fraction decomposition \cite{Kuznetsov_2011}.   It is known that the class of PH-distributions is dense in the class of all positive-valued distributions. Consequently, at least under a suitable condition, the scale function of any spectrally negative \lev process can be approximated by those of PH \lev processes.  



One major advantage of this approach is that the approximated scale function is given as a function in a closed form (usually a sum of complex exponentials), which enables one to analytically obtain other fluctuation identities explicitly.  This is particularly important when one needs to integrate some functional with respect to the resolvent measure, which is known to be written in terms of the scale function. This operation is needed, for example, when  overshoots/undershoots at down/up-crossing times are computed and when Carr's randomization (Canadization) is used to approximate the fluctuation identities killed at a deterministic time.


Another advantage of the PH-fitting approach is that the Laplace transform of the PH-distribution has an explicit expression. The numerical Laplace inversion approach inverts the equality written in terms of the Laplace exponents which do not in general admit analytical expressions (see \eqref{laplace_spectrally_negative}-\eqref{eq:scale} below).   For example, for the cases considered in Section \ref{numerics_EM} below, the associated Laplace exponents are not expressible in analytical form.  In other words, it contains two types of errors: (1) the approximation error caused while computing the Laplace exponent and (2) the error caused while inverting the Laplace transform.  On the other hand, the PH-fitting approach only contains the PH-fitting error thanks to the closed-form Laplace transform of the PH-distribution.

In order to evaluate the efficiency of the PH-fitting approach of approximating scale functions, we conduct a series of numerical experiments using the EM-algorithm of \cite{Asmussen_1996}.  
Using various fluctuation identities that can be computed analytically by the scale function, we compare their values of the fitted PH \lev processes with the simulated results.
While a number of hyperexponential fitting algorithms as in \cite{Feldmann_1998} for a distribution with a completely monotone density are guaranteed to converge to the desired limits, fitting of a general distribution is known to be difficult.  Nevertheless, at least for the case of finite \lev measures, our results show that the PH-fitting of scale functions are accurate even with a moderate number of phases and even for the case of uniformly-distributed jumps that are known to be difficult to fit. 


The rest of the paper is organized as follows.  Section \ref{section_phase_type_process} reviews the scale function and studies the scale function for the case the Laplace exponent is a rational function.  Section \ref{scale_phase} reviews the spectrally negative PH \lev process and studies its scale function.
We show via the continuity theorem that the PH-fitting approach can approximate the scale function for a general spectrally negative \lev process, and also discuss its practicability and limitations.
  In Section \ref{section_numerical_results}, we give the performance of the PH-fitting approach through a series of numerical experiments.

\section{Scale Functions}
 \label{section_phase_type_process}

Let $(\Omega,\mathcal{F},\mathbb{P})$ be a probability space hosting a \emph{spectrally negative} \lev process $X = \left\{X_t; t \geq 0 \right\}$, $\mathbb{P}^x$ the conditional probability under which $X_0 = x$ (and also $\mathbb{P} \equiv \mathbb{P}^0$), and $\mathbb{F} := \left\{ \mathcal{F}_t: t \geq 0 \right\}$ the filtration generated by $X$.  The process $X$ is uniquely characterized by its \emph{Laplace exponent}
\begin{align}
\psi(s)  := \log \E \left[ e^{s X_1} \right] =  c s +\frac{1}{2}\sigma^2 s^2 + \int_{(-\infty,0)} (e^{s z}-1 - s z 1_{\{z > -1\}}) \Pi (\diff z),  \label{laplace_spectrally_negative}
\end{align}
for any $s \geq 0$,
where $\Pi$ is a \lev measure with the support $(-\infty,0)$ and satisfies the integrability condition $\int_{(-\infty,0)} (1 \wedge |z|^2) \Pi(\diff z) < \infty$.  It has paths of bounded variation if and only if
\begin{align*}
\sigma = 0 \quad \textrm{and} \quad \int_{(-\infty,0)} (1 \wedge |z|) \Pi(\diff z) < \infty;
\end{align*}
see, for example, Lemma 2.12 of \cite{Kyprianou_2006}. In this case, we can rewrite the Laplace exponent (\ref{laplace_spectrally_negative}) by
\begin{align*}
\psi(s) = \mu s + \int_{(-\infty,0)} (e^{sz} - 1) \Pi (\diff z), 
\end{align*}
with
\begin{align*}
\mu := c - \int_{(-1,0)} z \Pi (\diff z).
\end{align*}
Here, we disregard the case when $X$ is a negative of a subordinator (or decreasing a.s.).

Fix $q \geq 0$ and any spectrally negative \lev process with its Laplace exponent $\psi$. The scale function $W^{(q)}: \R \rightarrow [0,\infty)$ is a function whose Laplace transform is given by
\begin{align}\label{eq:scale}
\int_0^\infty e^{-s x} W^{(q)}(x) \diff x = \frac 1
{\psi(s)-q}, \qquad s > \Phi(q)
\end{align}
where
\begin{align}
\Phi(q) :=\sup\{s  \geq 0: \psi(s)=q\}, \quad
q\ge 0. \label{zeta}
\end{align}
We assume $W^{(q)}(x)=0$ on $(-\infty,0)$.

Let us define the \emph{first down-} and \emph{up-crossing times}, respectively, by
\begin{align*}
\tau_a^- := \inf \left\{ t \geq 0: X_t < a \right\} \quad \textrm{and} \quad \tau_b^+ := \inf \left\{ t \geq 0: X_t >  b \right\}, \quad a,b \in \R,
\end{align*}
with $\inf \emptyset = \infty$.
Then we have for any $0 < x < b$
\begin{align}
\begin{split}
\E^x \left[ e^{-q \tau_b^+} 1_{\left\{ \tau_b^+ < \tau_0^-, \,  \tau_b^+ < \infty \right\}}\right] &= \frac {W^{(q)}(x)}  {W^{(q)}(b)}, \\
\E^x \left[ e^{-q \tau_0^-} 1_{\left\{ \tau_b^+ > \tau_0^-, \, \tau_0^- < \infty \right\}}\right] &= Z^{(q)}(x) -  Z^{(q)}(b) \frac {W^{(q)}(x)}  {W^{(q)}(b)} \end{split} \label{laplace_in_terms_of_z}
\end{align}
where
\begin{align*}
Z^{(q)} (x) := 1 + q  \int_0^x W^{(q)} (y) \diff y, \quad x \in \R. 
\end{align*}

Fix $a \geq 0$ and define $\psi_a(\cdot)$ as the Laplace exponent of $X$ under $\p_a$ with the change of measure 
\begin{align*}
\left. \frac {\diff \p_a} {\diff \p}\right|_{\mathcal{F}_t} = \exp(a X_t - \psi(a) t), \quad t \geq 0; 
\end{align*}
see page 213 of \cite{Kyprianou_2006}.
 Suppose $W_a^{(q)}$ and $Z_a^{(q)}$ are the scale functions associated with $X$ under $\p_a$ (or equivalently with $\psi_a(\cdot)$).  
Then, by Lemma 8.4 of \cite{Kyprianou_2006}, $W_a^{(q-\psi(a))}(x) = e^{-a x} W^{(q)}(x)$, $x \in \R$,
which is well-defined even for $q \leq \psi(a)$ by Lemmas 8.3 and 8.5 of \cite{Kyprianou_2006}.  In particular, we define
\begin{align*}
W_{\Phi(q)}(x) := W_{\Phi(q)}^{(0)}(x) = e^{-\Phi(q) x} W^{(q)}(x), \quad x \in \R,
\end{align*}
which is known to be monotonically increasing and
\begin{align*}
W_{\Phi(q)} (x) \nearrow  {(\psi'(\Phi(q)))^{-1}} \quad \textrm{as} \; x \rightarrow \infty,
\end{align*}
except for the case $q=0$ and $\psi'(\Phi(0)+) = 0$.
This also implies that the scale function $W^{(q)}$ increases exponentially in $x$.


Regarding the smoothness of the scale function, if the \lev measure has no atoms or the process has paths of unbounded variation, then $W^{(q)} \in C^1(0,\infty)$. If it has a Gaussian component ($\sigma > 0$), then $W^{(q)} \in C^2(0,\infty)$; see \cite{Chan_2009}.  In particular, a stronger result holds for the completely monotone jump case.  Recall that a density function $f$ is called \emph{completely monotone} if all the derivatives exist and, for every $n \geq 1$,
\begin{align*}
(-1)^n f^{(n)} (x) \geq 0, \quad x \geq 0,
\end{align*}
where $f^{(n)}$ denotes the $n^{th}$ derivative of $f$.

\begin{lemma}[Loeffen \cite{Loeffen_2008}] \label{lemma_completely_monotone}
If the (dual of the) \lev measure has a completely monotone density, then
$W_{\Phi(q)}'$ is also completely monotone.
\end{lemma}

Finally, the behavior in the neighborhood of zero is given as follows.
\begin{lemma} \label{lemma_zero}
For every $q \geq 0$, we have
\begin{align*}
W^{(q)} (0) &= \left\{ \begin{array}{ll} 0, & \textrm{if $X$ is of unbounded variation} \\ \frac 1 {\mu}, & \textrm{if $X$ is of bounded variation} \end{array} \right\}, \\
W^{(q)'} (0+) &= \left\{ \begin{array}{ll}  \frac 2 {\sigma^2}, & \textrm{if }\sigma > 0 \\   \infty, &  \textrm{if } \sigma = 0 \; \textrm{and} \; \Pi(-\infty,0) = \infty \\ \frac {q + \Pi(-\infty,0)} {\mu^2}, & \textrm{if $X$ is compound Poisson} \end{array} \right\}.
\end{align*}
\end{lemma}

\subsection{The case $\psi$ is rational}  As is discussed in \cite{Kuznetsov_2011}, when the Laplace exponent $\psi$ (extended to $\mathbb{C}$) is a rational function, or equivalently $X$ has jumps of rational transforms, we can invert \eqref{eq:scale} directly by partial fraction decomposition to obtain the scale function.  This class of processes is slightly more general than that of PH \lev processes we shall describe in the next section.  

Suppose $q \geq 0$ and $\psi'(0+) < 0$ if $q=0$.  Because $\psi(s) \xrightarrow{s \uparrow \infty} \infty$, this means that $(\psi(s)-q)^{-1}$ is a proper rational function that admits a partial fraction decomposition.  By  \eqref{zeta}, we can write
\begin{align}
\frac 1 {\psi(s)-q} = \frac {Q(s)} {(s-\Phi(q)) \prod_{i \in \mathcal{I}_q} (s+\xi_{i,q})}. \label{rational_proper}
\end{align}
Here $\mathcal{I}_q$ is the set of (the sign-changed) \emph{negative roots}:
\begin{align}
\mathcal{I}_q &:= \left\{ i: \psi (-\xi_{i,q}) = q \; \textrm{and} \; \mathcal{R} (\xi_{i,q}) > 0\right\}. \label{def_I_q}
\end{align}
The elements in $\mathcal{I}_q$ may not be
distinct; in this case, we take each as many times as its
multiplicity (see also Remark \ref{remark_multiplicity_unlikely} below). In addition, $Q$ is a polynomial that satisfies $Q(0) = \Phi(q) \prod_{i \in \mathcal{I}_q} \xi_{i,q} /q$ because $\psi(0) = 0$.

Let
$n$ denote the number of different roots in $\mathcal{I}_q$ and
$m_i$ denote the multiplicity of a root $\xi_{i,q}$ for $i = 1,\ldots,n$. We summarize the results given in Section 5.4 of \cite{Kuznetsov_2011}.
\begin{proposition} \label{proposition_main}
Suppose $q \geq 0$ and $\psi'(0+) < 0$ if $q=0$, and $\psi$ is a rational function such that \eqref{rational_proper} holds. Then the scale function is written
\begin{align}
W^{(q)}(x) =  \frac {e^{\Phi(q) x}} {\psi'(\Phi(q))}  - \sum_{i = 1}^n \sum_{k=1}^{m_i} B_{i,q}^{(k)} \frac {x^{k-1}} {(k-1)!} e^{-\xi_{i,q} x}, \quad x \geq 0, \label{scale_function}
\end{align}
where
\begin{align*}
B_{i,q}^{(k)} &:= \left. \frac 1 {(m_i-k)!} \frac {\partial^{m_i-k}} {\partial s^{m_i-k}} \frac { (s+\xi_{i,q})^{m_i}} {q-\psi(s)} \right|_{s = -\xi_{i,q}}, \quad 1 \leq k \leq m_i \; \textrm{and} \; 1 \leq i \leq n.
\end{align*}
In particular, if all the roots in $\mathcal{I}_q$ are distinct, then
\begin{align}
W^{(q)}(x) =  \frac {e^{\Phi(q) x}} {\psi'(\Phi(q))}  - \sum_{i = 1}^n  B_{i,q} e^{-\xi_{i,q} x}, \quad x \geq 0, \label{scale_function_distinct}
\end{align}
where
\begin{align*}
B_{i,q} &:= \left. \frac { s+\xi_{i,q}} {q-\psi(s)} \right|_{s = -\xi_{i,q}} = - \frac 1 {\psi'(-\xi_{i,q})}.
\end{align*}
\end{proposition}

\begin{remark}  At $x=0$, we have
\begin{align*}
W^{(q)}(0) =  \frac 1 {\psi'(\Phi(q))} - \sum_{i = 1}^n B_{i,q},
\end{align*}
which, by Lemma \ref{lemma_zero}, vanishes when $X$ is of unbounded variation while it is $\mu^{-1}$ otherwise.
\end{remark}


%


\section{Scale functions for Spectrally negative Phase-type \lev processes}
\label{scale_phase}


Consider a continuous-time Markov chain $Y = \{ Y_t; t \geq 0 \}$ with a finite
state space $\{1,\ldots,m \} \cup \{ \Delta \}$ where $1,\ldots,m$
are transient and $\Delta$ is absorbing. Its initial distribution is given by
a simplex ${\bm \alpha}=[\alpha_1, \ldots, \alpha_m]$
such that $\alpha_i=\p \left\{ Y_0=i \right\}$ for every $i = 1,\ldots,m$.  The intensity matrix ${\bm Q}$ is partitioned into the $m$ transient
states and the absorbing state $\Delta$, and is given by
\begin{align*}
{\bm Q}  := \begin{bmatrix} {\bm T} & {\bm t} \\ {\bm 0} & 0 \end{bmatrix}.
\end{align*}
Here  ${\bm T}$  is an  $m \times m$-matrix called the PH-generator, and ${\bm t} = - {\bm T} {\bm 1}$ where  ${\bm 1} =
[1,\ldots,1]'$. A distribution is called PH with
representation $(m, {\bm \alpha}, {\bm T})$ if it is the
distribution of the absorption time to $\Delta$ in the Markov chain
described above. It is known that ${\bm T}$ is non-singular and thus
invertible; see \cite{Asmussen_1996}.  Its distribution and density functions
are given, respectively, by
\begin{align*}
F(z; \bm \alpha, \bm T) =  1-{\bm \alpha} e^{{\bm T} z} {\bm 1} \quad \textrm{and} \quad f(z; \bm \alpha, \bm T) = {\bm \alpha} e^{{\bm T} z} {\bm t}, \quad z > 0.
\end{align*}

Let $X = \left\{X_t; t \geq 0 \right\}$ be a spectrally negative \lev process of the form
\begin{equation}
  X_t  - X_0=\mu t+\sigma B_t - \sum_{n=1}^{N_t} Z_n, \quad 0\le t <\infty, \label{levy_canonical}
\end{equation}
for some $\mu \in \R$ and $\sigma \geq 0$.  Here $B=\{B_t; t\ge 0\}$ is a standard Brownian motion, $N=\{N_t; t\ge 0\}$ is a Poisson process with arrival rate $\lambda$, and  $Z = \left\{ Z_n; n = 1,2,\ldots \right\}$ is an i.i.d.\ sequence of PH-distributed random variables with representation $(m,{\bm \alpha},{\bm T})$. These processes are assumed mutually independent. Its Laplace exponent is then
\begin{align*}
 \psi(s)   = \mu s + \frac 1 2 \sigma^2 s^2 + \lambda \left( {\bm \alpha} (s {\bm I} - {\bm{T}})^{-1} {\bm t} -1 \right),
 \end{align*}
which is analytic for every $s \in \mathbb{C}$ except at the eigenvalues of ${\bm T}$.

We shall see that the scale function of this process is a special case of the ones given in Proposition  \ref{proposition_main}.  Note that the case  all the roots in $\mathcal{I}_q$ are distinct has been studied by \cite{Kyprianou_Palmowski_2007} when $q=0$ and $\psi'(0+) > 0$.  More specialized cases with hyperexponential and Erlang-type jumps are given in \cite{Avram_2004, Loeffen_2008}.

Disregarding the negative subordinator case, we consider the following two cases:
\begin{enumerate}
\item[] \textbf{Case 1:} when $\sigma > 0$,
\item[]  \textbf{Case 2:} when $\sigma = 0$ and $\mu > 0$ (i.e.\ $X$ is a compound Poisson process).
\end{enumerate}
Here, in \textbf{Case 2}, down-crossing of a threshold can occur only by jumps; see, for
example, Chapter III of \cite{Bertoin_1996}.  On the other
hand, in \textbf{Case 1}, down-crossing can occur also by \emph{creeping
downward} (by the diffusion components). 

Fix $q > 0$. Recall \eqref{def_I_q}, and further define  the set of (the sign-changed) \emph{negative poles}:
\begin{align*}
\mathcal{J}_q &:= \left\{ j: \frac q {q - \psi(-\eta_j)} = 0 \; \textrm{and} \; \mathcal{R} (\eta_j) > 0\right\}.
\end{align*}
As is the case for $\mathcal{I}_q$ , the elements in $\mathcal{J}_q$ may not be
distinct, and, in this case, we take each as many times as its
multiplicity. By Lemma 1 of
\cite{Asmussen_2004}, we have
\begin{align*}
|\mathcal{I}_q| = \left\{ \begin{array}{ll} |\mathcal{J}_q| + 1, & \textrm{for \textbf{Case 1}}, \\ |\mathcal{J}_q|, & \textrm{for \textbf{Case 2}}. \end{array} \right.
\end{align*}
In particular, if the representation is minimal (see \cite{Asmussen_2004}), we have $|\mathcal{J}_q| = m$.

Let $\textbf{e}_q$ be an independent exponential random variable with parameter $q$ and denote the \emph{running supremum} and \emph{infimum}, respectively, by
\begin{align*}
\overline{X}_t := \sup_{0 \leq s \leq t} X_s \quad \textrm{and} \quad \underline{X}_t := \inf_{0 \leq s \leq t} X_s, \quad t \geq 0.
\end{align*}
The \emph{Wiener-Hopf factorization} states that $q /{(q - \psi(s))} = \varphi_q^+ (s) \varphi_q^- (s)$ for every $s \in \mathbb{C}$ such that $\mathcal{R}(s) = 0$, with the
\emph{Wiener-Hopf factors}
\begin{align*}
\varphi_q^- (s) := \E \left[ \exp (s \underline{X}_{\textbf{e}_q}) \right] \quad \textrm{and} \quad \varphi_q^+ (s) := \E \left[ \exp (s \overline{X}_{\textbf{e}_q}) \right] 
\end{align*}
that are analytic for $s$ with $\mathcal{R}(s) > 0$ and $\mathcal{R}(s) < 0$, respectively.  For the case of spectrally negative \lev processes, $\varphi_q^+ (s) = \Phi(q)/(\Phi(q)-s)$; see page 213 of \cite{Kyprianou_2006}.  
Moreover, by Lemma 1 of \cite{Asmussen_2004}, we have, for every $s$ such that $\mathcal{R}(s) > 0$,
\begin{align*}
\varphi_q^- (s) = \frac {\prod_{j \in \mathcal{J}_q} (s+\eta_j)} {\prod_{j \in \mathcal{J}_q} \eta_j} \frac {\prod_{i \in \mathcal{I}_q} \xi_{i,q}} {\prod_{i \in \mathcal{I}_q} (s+\xi_{i,q})}. 
\end{align*}
Hence \eqref{rational_proper} holds by setting
\begin{align*}
Q(s) = \frac {\Phi(q)} q \frac {\prod_{j \in \mathcal{J}_q} (s+\eta_j) \prod_{i \in \mathcal{I}_q} \xi_{i,q}} {\prod_{j \in \mathcal{J}_q} \eta_j}. 
\end{align*}
Consequently, the scale function can be written as \eqref{scale_function} or \eqref{scale_function_distinct} in Proposition \ref{proposition_main}.


\begin{remark} \label{remark_multiplicity_unlikely}As is discussed in Section 5.4 of \cite{Kuznetsov_2011}, it is in fact highly unlikely that any root in $\mathcal{I}_q$ has multiplicity larger than one.  This implies that the scale function is most likely simplified to \eqref{scale_function_distinct}.  This fact is confirmed in Section \ref{section_numerical_results} where, in all the numerical examples considered, all the roots in $\mathcal{I}_q$ turn out to be distinct.
\end{remark}

\begin{example}[Hyperexponential Case] \label{example_hyperexponential}
As an important example where all the roots in $\mathcal{I}_q$ are distinct and real, we consider the case where $Z$ has a hyperexponential distribution with a density function
\begin{align*}
f (z)  = \sum_{j=1}^m p_j \eta_j e^{- \eta_j z}, \quad z > 0,
\end{align*}
for some $0 < \eta_1 < \cdots < \eta_m < \infty$ and $p_j > 0$ for $1 \leq j \leq m$ such that $p_1 + \cdots + p_m = 1$.  Its Laplace exponent (\ref{laplace_spectrally_negative})
is then
\begin{align*}
\psi(s) = \mu s + \frac 1 2 \sigma^2 s^2  - \lambda
\sum_{j=1}^m p_j \frac s {\eta_j + s}.
\end{align*}
Notice in this case that $-\eta_1$, \ldots, $-\eta_m$ are the poles of the Laplace exponent.  Furthermore, all the roots in $\mathcal{I}_q$ are distinct and real and satisfy the following interlacing condition for every $q > 0$:
\begin{enumerate}
\item
for \textbf{Case 1}, there are $m+1$ roots $-\xi_{1,q}, \ldots, -\xi_{m+1,q}$ such that
\begin{align}
0 < \xi_{1,q} < \eta_1 < \xi_{2,q} < \cdots < \eta_m < \xi_{m+1,q} < \infty; \label{interlacing1}
\end{align}
\item for \textbf{Case 2}, there are $m$ roots $-\xi_{1,q}, \ldots, -\xi_{m,q}$  such that
\begin{align}
0 < \xi_{1,q} < \eta_1 < \xi_{2,q} < \cdots <  \xi_{m,q}  < \eta_m < \infty. \label{interlacing2}
\end{align}
\end{enumerate}
Because all roots are real and distinct, the scale function can be written as \eqref{scale_function_distinct}.

The class of hyperexponential distributions is important as it is dense in the class of all positive-valued distributions with completely monotone densities.  We refer the reader to \cite{Albrecher_2010, Feldmann_1998, Kammler_1976} for approximation methods.
\end{example}

\begin{example}[Coxian Case] \label{example_coxian}
The Coxian distributionon is a special case of the PH-distribution  where its  PH-generator $\bm T$ has the form
\begin{align*}
\bm T = \left[ \begin{array}{ccccccc} -\eta_1 & p_1 \eta_1 & 0 & \cdots & 0 & 0 & 0 \\ 0 &-\eta_2 & p_2 \eta_2 & \cdots & 0 & 0 & 0 \\ 0 &0 & -\eta_3 & \cdots & 0 & 0 & 0 \\  \cdots  &\cdots & \cdots & \cdots & \cdots & \cdots & \cdots \\ 0 &0 & 0 & \cdots & 0 & -\eta_{m-1} & p_{m-1}\eta_{m-1} \\ 0 &0 & 0 & \cdots & 0 & 0 & -\eta_m \end{array}\right],
\end{align*}
for some $p_1, \ldots, p_{m-1} \in (0, 1]$ and $\bm \alpha = [1,0, \ldots, 0]$.  

Despite its simple structure,  it is almost as general as the PH-distribution because any acyclic PH-distribution has an equivalent Coxian representation \cite{Cumani_1982, Dehon_1982}.  Moreover, due to the sparsity of the PH-generator, it is numerically easier to compute the roots and poles in $\mathcal{I}_q$ and $\mathcal{J}_q$ even for large $m$.  In Section \ref{section_numerical_results}, we give comparisons between fitting general PH-distributions and fitting Coxian distributions.

\end{example}

\subsection{Approximation of the scale function of a general spectrally negative \lev process} \label{section_approximation}
Under a suitable assumption, the scale function obtained above can be used to approximate the
scale function of a general spectrally negative \lev process.  

By
Proposition 1 of \cite{Asmussen_2004}, there exists,
for any spectrally negative \lev process $X$, a sequence of
spectrally negative PH \lev processes  $X^{(n)}$
converging to $X$ in $D[0,\infty)$.  This is equivalent to saying
that $X_1^{(n)} \rightarrow X_1$ in distribution by Corollary VII 3.6 of \cite{Jacod_Shirayev_2003}; see also \cite{Pistorius_2006}.   Suppose $\psi^{(n)}$ (resp.\ $\psi$) and
$W^{(q,n)}$ (resp.\ $W^{(q)}$) are the Laplace exponent and the scale
function of $X^{(n)}$ (resp.\ $X$).

\begin{proposition}  \label{proposition_convergence}If the jump parts of $X^{(n)}$ and $X$ have paths of bounded variation with the common Gaussian coefficient $\sigma \geq 0$.  Then, $W^{(q,n)}(x) \rightarrow W^{(q)}(x)$ as $n \uparrow \infty$ for every $x \geq 0$. 
\end{proposition}
\begin{proof}
Because $W^{(q)}$ is an increasing function, the measure $W^{(q)}(\diff x)$ associated with the distribution of $W^{(q)}(0,x]$ is well-defined and we obtain as in page 218 of \cite{Kyprianou_2006},
\begin{align}
\int_{[0,\infty)} e^{-s x} W^{(q)}(\diff x) = \frac s {\psi(s)-q}. \label{scale_version_modified}
\end{align}
By assumption,  both $X^{(n)}$ and $X$ can be decomposed into a difference between two subordinators  plus a Brownian motion.  Hence, the convergence in distribution of $X_1^{(n)}$ to $X_1$  implies $\psi^{(n)}(s) \rightarrow \psi(s)$ for every $s > 0$.  Now in view of
(\ref{scale_version_modified}), the convergence of the scale function holds by the continuity of the scale function and the
continuity theorem; see \cite{Feller_1971}, Theorem 2a,
XIII.1.  
\end{proof}

This proposition does not directly imply the same results when the jumps are of unbounded variation.   However, at least in principle, \eqref{scale_version_modified} can be approximated by that of a spectrally negative PH \lev process.  First, because \eqref{scale_version_modified} converges to zero as $s \rightarrow \infty$ for the case of unbounded variation in view of Lemma \ref{lemma_zero}, the main issue essentially is for $s$ on compacts on condition that the approximation can be done by those of unbounded variation; we revisit this issue in Section  \ref{subsection_contributions_limitations}.  Now consider splitting $X$, for small $\varepsilon > 0$, into the sum of $X^{(\varepsilon,1)} + X^{(\varepsilon,0)}$ whose Laplace exponents are
\begin{align} \label{equation_splitting}
\begin{split}
\psi_{\varepsilon,1}(s)  &:=  c s +\frac{1}{2}\sigma^2 s^2 + \int_{(-\infty,-\varepsilon]} (e^{s z}-1 - s z 1_{\{z > -1\}}) \Pi (\diff z),  \\
\psi_{\varepsilon,0}(s)  &:=  \int_{(-\varepsilon,0)} (e^{s z}-1 - s z 1_{\{z > -1\}}) \Pi (\diff z),  
\end{split}
\end{align}
respectively.  The former has jumps of bounded variation and hence by Proposition \ref{proposition_convergence} we can construct a sequence of Laplace exponents $(\psi_{\varepsilon,1}^{(n)})_{n \geq 1}$ of the PH \lev processes converging to it.  For the latter, for sufficiently small $\varepsilon$ (for $s$ on compacts),
\begin{align*}
\psi_{\varepsilon,0}(s)  &\approx  \frac {s^2} 2 \int_{(-\varepsilon,0)} z^2\Pi (\diff z),  
\end{align*}
which can be approximated by that of the Brownian motion with a Gaussian coefficient $\sigma^{\varepsilon} := (\int_{(-\varepsilon,0)} z^2\Pi (\diff z))^{1/2}$.  This implies that $\psi(s)$ and hence  \eqref{scale_version_modified} as well can be approximated by  $\psi_{\varepsilon,1}^{(n)}(s) + \frac {\sigma^\varepsilon} 2 s^2$ (corresponding to PH \lev processes) at least on compacts if we choose $\varepsilon$ sufficiently small.

\subsection{Contributions and Limitations} \label{subsection_contributions_limitations}  In view of Proposition \ref{proposition_convergence} above, at least in principle, a scale function can be approximated by that of the PH \lev process given that the latter can be computed. This approach certainly has both pros and cons.  Here, we conclude this section by discussing its contributions and  limitations; in the next section we further evaluate it numerically.

\textbf{Contributions}.  The main advantage of this approach is due to its explicit form as in \eqref{scale_function} and \eqref{scale_function_distinct}.   One important application of the scale function is its expression of the resolvent measure of a \lev process or its reflected process; see, for example, \cite{Kyprianou_2006,Pistorius_2004}.  For example, we can write
\begin{align}
U_A(x) := \E^{x}\left[ \int_0^{\tau_{A}^-} e^{-qt} f(X_t) \diff t \right] &=
W^{(q)} (x-A) \Psi_f (A) -
\Theta_f(x;A), \quad x \in \R, \label{eq_resolvent}
\end{align}
for any measurable function $f$ satisfying $\int_0^\infty e^{-\lapinv y} |f(y+A) | \diff y < \infty$ for any $A \in \R$,
where
\begin{align*}
\Psi_f(A) := \int_0^\infty e^{-\lapinv y} f(y+A)\diff y \quad \textrm{and} \quad
\Theta_f(x; A) := \left\{ \begin{array}{ll} \int_{A}^x  W^{(q)}(x-y) f(y) \diff y, & x > A, \\ 0, & x \leq A; \end{array} \right.
\end{align*}
for the derivation, see, for example, \cite{Egami-Yamazaki-2011} and \cite{Yamazaki_2012}.  

Here potential difficulty lies in the computation of $\Theta_f(x; A)$.   When  the integral $\int_A^\infty e^{-\beta x} f(x) \diff x$ can be written analytically, it can be approximated directly by numerical Laplace inversion via the convolution theorem.  However, this can be difficult depending on  the form of the function $f$.  There are 
examples where the integral with respect to the resolvent \eqref{eq_resolvent} needs to be computed repeatedly (and replacing $f$ with it), and hence \eqref{eq_resolvent} needs to be computed for the function $f$ that  itself is dependent on the scale function.
For example, in Carr's randomization (Canadization) method, one wants to approximate the value function for a constant finite time horizon problem with that of an Erlang distributed time horizon problem for a sufficiently large shape parameter $k$; see, for example, \cite{Kyprianou_Pistorius_2003}.   This requires applying repeatedly (for $k$ times) the integration with respect to the resolvent measure as in \eqref{eq_resolvent}.  For the PH case, the resulting value function can be obtained explicitly in many cases; otherwise, the computation is practically infeasible unless $k$ is small.  Another example where the PH fitting may be more suitable than numerical Laplace inversion  is the case where
the function $f$ is written in terms of the \lev measure, e.g.\ the Gerber-Shiu function.

%
%
%
%

\textbf{Limitations.} These advantages can be enjoyed only on condition that the scale function can be approximated accurately; this is directly dependent on how a \lev measure is approximated by a PH-distribution (times a finite constant).  Unfortunately,  except for the case the \lev measure has a completely monotone density, there does not currently exist a PH-fitting algorithm that always converges and works for arbitrary distributions.  Because the numerical Laplace inversion approach is known to work with high speed and high accuracy, this is clearly a major drawback of the PH-fitting approach. 

A particular weak point arises when the jump part of the process to be approximated is of infinite activity or of unbounded variation. Under the PH-fitting, these infinitesimal jumps must be approximated either by compound Poisson processes and/or Brownian motions.  In particular, when $\sigma = 0$ and $\Pi(0,\infty) = \infty$, $W^{(q)'}(0+) = \infty$ by Remark \ref{lemma_zero} and hence it is expected to be difficult to fit in these cases.

When the approximation is done by compound Poisson processes, we need to take the jump intensity $\lambda$ arbitrarily high.  Indeed one needs to first choose a  small truncation parameter $\varepsilon$ and $\lambda = \Pi(\varepsilon, \infty)$ and then fit a probability distribution to $\Pi(\varepsilon, \cdot)/\lambda$. Consequently, the error of approximating $\Pi(\varepsilon, \cdot)$ directly depends on how large $\lambda$ is.   We therefore face the tradeoff between minimizing the truncation error (by choosing $\varepsilon$ small) and minimizing the error associated with approximating $\Pi(\varepsilon, \cdot)$.

As we have discussed above, when the jumps are of unbounded variation, the use of Brownian motion by the decomposition \eqref{equation_splitting} can also be considered.  However, we also face a similar issue of choosing the value of $\varepsilon$ in view of \eqref{equation_splitting}.  In particular, we expect that the approximation would be more difficult for the case of unbounded variation with $\sigma = 0$. For the case with  $\sigma > 0$, the Gaussian coefficients of approximating processes are at least $\sigma$; consequently \eqref{scale_version_modified} converges to zero (uniformly) as $s \rightarrow \infty$ and  hence we can focus on $s$ on compacts. On the other hand, this is not the case when $\sigma = 0$ because the Gaussian coefficients of the approximating processes are $\sigma^\varepsilon$ that vanish as $\varepsilon \rightarrow 0$.  


We remark here that these  issues may be resolved for the case the \lev measure to be approximated has a completely monotone density.  For the finite activity case, the approximation can be done by hyperexponential \lev processes. Otherwise, it can be done by meromorphic \lev processes \cite{Kuznetsov_2011} (with some truncation), whose scale function is a slight modification of \eqref{scale_function_distinct} with $n = \infty$.  The fitting is known to be fast and accurate, and thanks to the interlacing condition as in \eqref{interlacing1} and \eqref{interlacing2}, the solutions to $\psi(\cdot)=q$ can be obtained easily.

\section{Numerical Experiments} \label{section_numerical_results}
In this section, we illustrate numerically the effectiveness of the PH-fitting of scale functions through a series of numerical experiments.    For each \lev measure we shall consider below, the EM-algorithm of \cite{Asmussen_1996} is used to fit PH-distributions;  fitted PH \lev processes and their scale functions are then constructed.  We evaluate the accuracy of the PH-fitting approximation by comparing some fluctuation identities approximated by the fitted scale functions  and those approximated by simulation.  Because it is widely known that fitting a distribution with a completely monotone density can be easily done with high accuracy, here we focus on approximating the \lev process whose \lev density \emph{is not} completely monotone.  

\subsection{The EM-algorithm} 
For any arbitrary non-negative continuous distribution, the EM-algorithm approximates it by constructing a sequence of parameter estimates $(m,{\bm \alpha^{(k)}, \bm T^{(k)}}; k \geq 0)$ for fixed number of phases $m$. 
For a fixed non-negative distribution with density $h$ (to be approximated) and a PH-distribution with  $(m,{\bm \alpha, \bm T})$, the \emph{Kullback-Leibler divergence} is given by
\begin{align*}
\int_0^\infty \log \frac {h(x)} {f(x; \bm \alpha, \bm T)} h(x) \diff x = \int_0^\infty \log (h(x)) h(x) \diff x - \int_0^\infty  \log (f(x; \bm \alpha, \bm T)) h(x) \diff x,
\end{align*}
which is non-negative and equals zero if and only if $h(\cdot)=f(\cdot; \bm \alpha, \bm T)$ Lebesgue-a.e.  The idea is to obtain $(\bm \alpha, \bm T)$ such that this is minimized.  Because the first term on the right-hand side depends only on the given density $h$, it is equivalent to maximizing 
\begin{align*}
\delta(\bm \alpha, \bm T; h) := \int_0^\infty  \log (f(x; \bm \alpha, \bm T)) h(x) \diff x.
\end{align*} Instead of doing so directly, the EM-algorithm first generates $({\bm \alpha^{(0)}, \bm T^{(0)}})$ randomly and repeats the so-called \emph{EM-step} to construct $({\bm \alpha^{(k+1)}, \bm T^{(k+1)}})$ from $({\bm \alpha^{(k)}, \bm T^{(k)}})$.
This step consists of evaluating the conditional expectation (\emph{E-step}) and maximizing it (\emph{M-step}).   Recall that a PH-distribution corresponds to that of an absorption time of a continuous-time Markov chain $Y$; see Section \ref{scale_phase}. In the E-step, the conditional expectation of the sufficient statistic $S$ of the multi-parameter exponential family
\begin{align}
\int_0^\infty E[S | x, {\bm \alpha^{(k)},  \bm T^{(k)}}] h(x) \diff x \label{expectation_statistic}
\end{align}
is computed.   Here $S$ consists of the random variables representing (1) the number of times $Y$ starts in each state, (2) the length of time $Y$ spends in each state and (3) the number of jumps $Y$ makes between any combination of states.   The expectation $E$ is under the condition that $Y$ is a continuous-time Markov chain with  initial distribution $\bm \alpha^{(k)}$ and transition matrix $\bm T^{(k)}$ and its absorption time equals $x$.  In the M-step, a new estimate $({\bm \alpha^{(k+1)},  \bm T^{(k+1)}})$ is computed. For a more detailed description of the EM-algorithm, we refer the reader to \cite{Asmussen_1996}.

By construction (due to Jensen's inequality), it is ensured that
\begin{align*}
\delta({\bm \alpha^{(k+1)},  \bm T^{(k+1)}};h )\geq \delta({\bm \alpha^{(k)}, T^{(k)}}; h), \quad k \geq 0,
\end{align*}
and hence it converges.  Ideally, the limit is the desired maximum likelihood estimates, but unfortunately this is not guaranteed; as discussed in Dempster et al.\ \cite{Dempster_1977} and Wu \cite{Wu_1983}, it may converge to local maxima or even saddle points.  Moreover, the performance clearly depends on the fixed number of phases $m$.  It is therefore important to test the algorithm for various examples of $h$ and for various values of $m$.

For our numerical results, we use EMpht (with slight modification for our purpose) which is written in C and is publicly available\footnote{Available at http://home.imf.au.dk/asmus/pspapers.html as of July 24, 2013.}.   The program is capable of fitting PH-distributions to a sample or another given continuous distribution.   Because our objective here is to evaluate the accuracy of the PH-fitting algorithm, we focus on the latter.   EMpht gives a sequence of estimates as addressed above.  In particular, \eqref{expectation_statistic} is approximated by its discretization with each interval less than $0.05$ and the probability mass in each interval less than $0.01$.  Its support is also truncated to $[0,25]$. For all numerical results given in this section, we use Windows 7, Intel Xeon CPU E$5$-$2620$, $2.00$GHz and $24.0$GB of RAM.  Except for the EMpht algorithm (used to obtain the approximation of the PH distribution), all the codes are written and run in MATLAB.  Some of the fitted PH-distributions and the parameters of the scale functions are given in the appendix.

\begin{table}　
\begin{tabular}{|c|rr|rr|rr|rr|}
\hline
$m$ & \multicolumn{2}{c|}{(i) Normal} & \multicolumn{2}{c|}{(ii) Weibull} & \multicolumn{2}{c|}{(iii) Lognormal}  & \multicolumn{2}{c|}{(iv) Uniform} \\
\hline 
 3   &  2.19 &(1010)  &   0.70 &(270)& 0.68 &(280) & 0.23 &(360) \\
 6   &  25.43 &(2040)& 65.35 &(3880)  & 19.97 &(1210)  & 3.97 &(1420)\\
 9   &  56.41 &(1430) &  117.73 &(2540) & 183.94 &(2730)  & 55.36 &(4370)\\
 12 &  159.57 &(1330)   & 251.54 &(2280)  &  1846.03 &(9130)& 777.13 &(9940)\\
  15 &  1612.30 &(2260)  & 900.55 &(3290)  & 1868.46 &(4820)   & 6678.44  &(24070)\\
  \hline
\end{tabular} \\
Regular Fit
\vspace{0.2cm}

\begin{tabular}{|c|rr|rr|rr|rr|}
\hline
$m$ & \multicolumn{2}{c|}{(i) Normal} & \multicolumn{2}{c|}{(ii) Weibull} & \multicolumn{2}{c|}{(iii) Lognormal}  & \multicolumn{2}{c|}{(iv) Uniform} \\
\hline 
 3   &  1.45 &(580) &   0.76 & (280)& 0.45 &(160) & 0.19 &(260) \\
 6   &  45.22 &(1710) & 49.71 &(2310)  & 126.61 &(2900) & 0.83 &(470)\\
 9   &  204.92 &(1250)  &  552.06 &(2890) & 1950.75 &(4840) & 8.31 &(1500)\\
 12  &  360.72 &(810) & 1433.51& (2030)  &  801.71& (3840)& 18.61 &(1330)\\
  15 &  878.76 &(810)   & 4476.22 &(2310) & 4264.34 &(7310)  & 39.12& (1260)\\
  \hline
\end{tabular} \\
Coxian Fit
\vspace{0.2cm}
\caption{Computation time (in seconds) and the number of EM-steps (in parentheses) required to compute the approximating PH-distributions.} \label{computation_time}
\end{table}

\begin{table}  
\begin{tabular}{|c|rr|rr|rr|rr|}
\hline
 & \multicolumn{2}{c|}{(i) Normal} & \multicolumn{2}{c|}{(ii) Weibull} & \multicolumn{2}{c|}{(iii) Lognormal} & \multicolumn{2}{c|}{(iv) Uniform} \\
\hline 
 $m$  &    \multicolumn{1}{c}{Regular} &  \multicolumn{1}{c|}{Coxian}&  \multicolumn{1}{c}{Regular} &  \multicolumn{1}{c|}{Coxian} &   \multicolumn{1}{c}{Regular} &  \multicolumn{1}{c|}{Coxian} &  \multicolumn{1}{c}{Regular} &  \multicolumn{1}{c|}{Coxian} \\
\hline
 3   &   0.330906 & 0.338259 & 0.608658 &0.348224 &  0.510884 &0.257411& 0.641700 &0.353807 \\
 9   &  1.949233 & 1.080589 & 2.180710 &1.250907 &  8.767836 &0.375460 & 1.629646 &1.190153\\
  15   & 55.028208 &2.434088  & 44.320027 &2.779834  &  72.054121 &0.521596 & 27.589086& 2.795609\\
  \hline
\end{tabular} \\
(a) $\sigma = 1$ and $\lambda = 5$ 
\vspace{0.2cm}
\\
\begin{tabular}{|c|rr|rr|rr|rr|}
\hline
 & \multicolumn{2}{c|}{(i) Normal} & \multicolumn{2}{c|}{(ii) Weibull} & \multicolumn{2}{c|}{(iii) Lognormal} & \multicolumn{2}{c|}{(iv) Uniform} \\
\hline 
 $m$  &    \multicolumn{1}{c}{Regular} &  \multicolumn{1}{c|}{Coxian}&  \multicolumn{1}{c}{Regular} &  \multicolumn{1}{c|}{Coxian} &   \multicolumn{1}{c}{Regular} &  \multicolumn{1}{c|}{Coxian} &  \multicolumn{1}{c}{Regular} &  \multicolumn{1}{c|}{Coxian} \\
\hline
 3   &  0.545884  &  0.313348 & 0.485286 &0.275546 & 0.436516  &0.244288 & 0.668737 & 0.306976 \\
 9   & 2.105709  & 1.164537 & 2.136692 & 1.108915 &1.954952   & 0.432049& 2.788000 & 0.980335\\
  15   &  55.963374 &  2.610069 & 43.190965 &2.492024  & 71.652231  & 0.768451 & 32.040391 & 2.493620\\
  \hline
\end{tabular} \\
(b) $\sigma = 0$ and $\lambda = 5$
\vspace{0.2cm}
  \\
\begin{tabular}{|c|rr|rr|rr|rr|}
\hline
 & \multicolumn{2}{c|}{(i) Normal} & \multicolumn{2}{c|}{(ii) Weibull} & \multicolumn{2}{c|}{(iii) Lognormal} & \multicolumn{2}{c|}{(iv) Uniform} \\
\hline 
 $m$  &    \multicolumn{1}{c}{Regular} &  \multicolumn{1}{c|}{Coxian}&  \multicolumn{1}{c}{Regular} &  \multicolumn{1}{c|}{Coxian} &   \multicolumn{1}{c}{Regular} &  \multicolumn{1}{c|}{Coxian} &  \multicolumn{1}{c}{Regular} &  \multicolumn{1}{c|}{Coxian} \\
\hline
 3   &  0.477291  &  0.232331 & 0.455249 & 0.244142 & 0.547964 & 0.237181 & 0.538500 & 0.265110\\
 9   &  1.470776 & 0.504720 & 1.494699 &0.496120 & 1.821302 & 0.427017 & 1.901286& 0.415698 \\
  15   & 54.557099 &  0.758591 & 42.471325 & 0.777695 & 72.913717 & 0.704831 & 27.037137& 0.655618 \\
  \hline
\end{tabular} \\
(c) $\sigma = 1$ and $\lambda = 10$ 
\caption{Computation time (in seconds) to compute the parameters of the scale functions.} \label{computation_time_scale_function_ruin}
\end{table}

%

\begin{figure}[htbp]
\begin{center}
\begin{minipage}{1.0\textwidth}
\centering
\begin{tabular}{cc}
\includegraphics[scale=0.6]{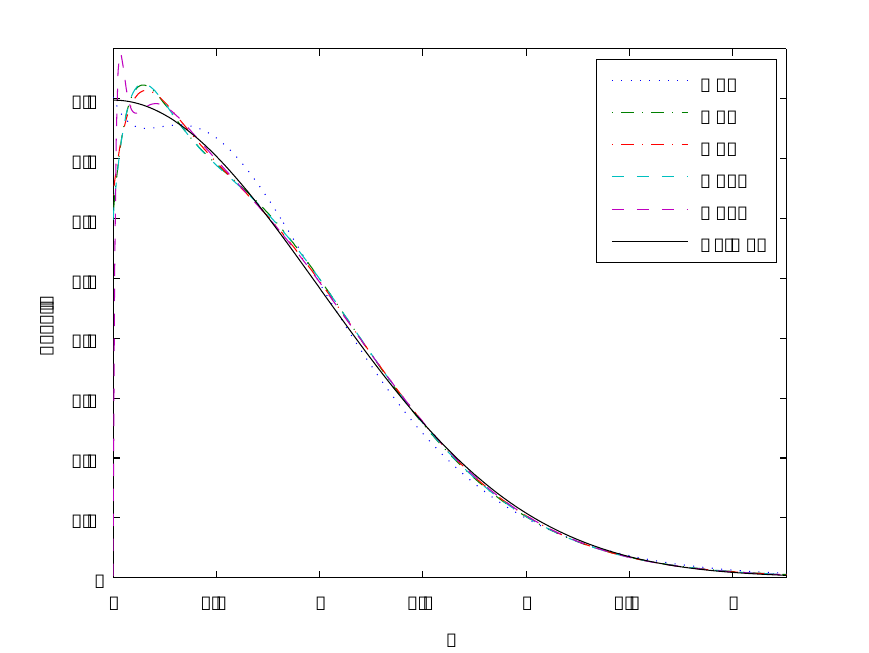}  & \includegraphics[scale=0.6]{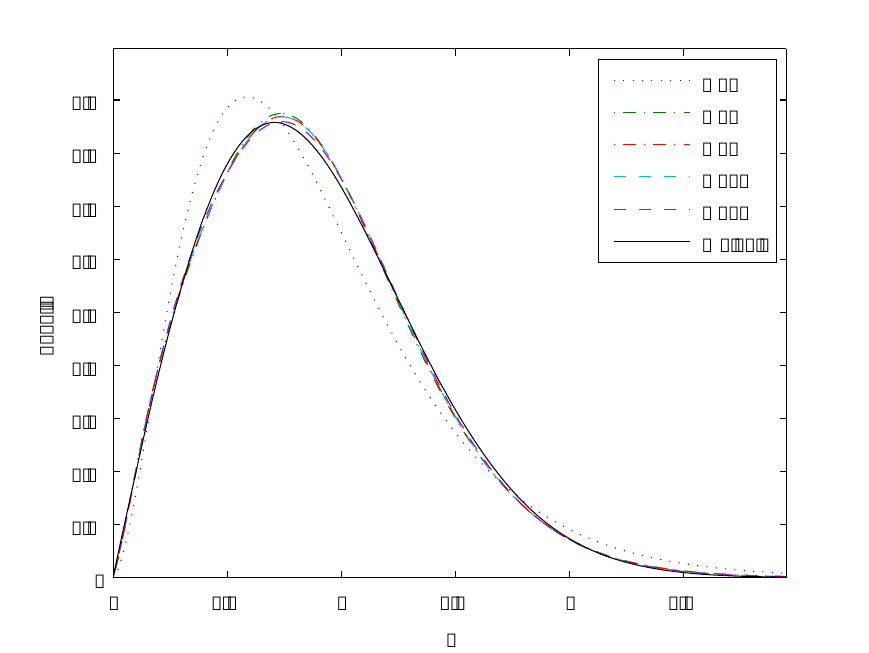}  \\
(i) Normal$(0,1)$ &  (ii) Weibull$(2,1)$ \\
\includegraphics[scale=0.6]{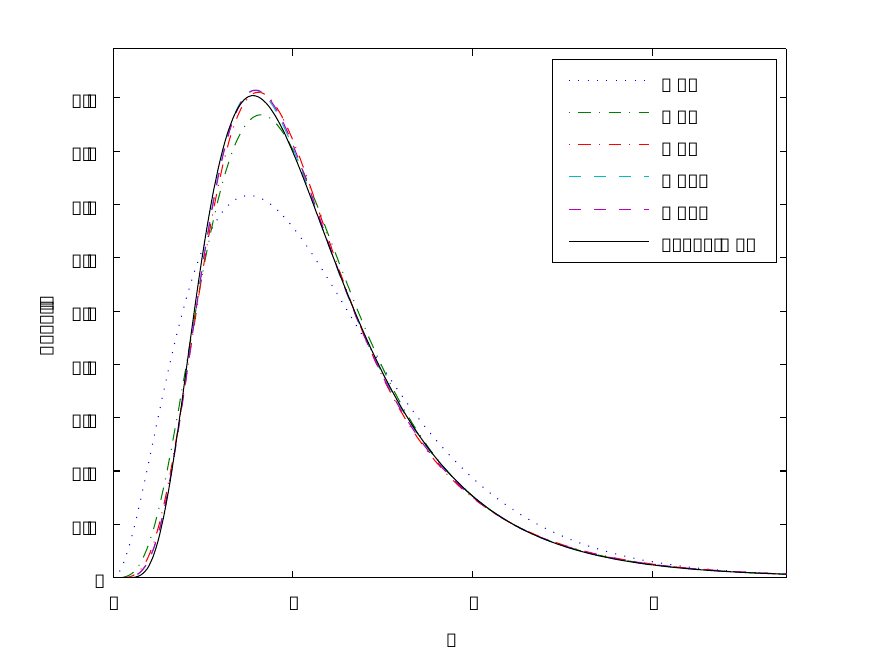}  & \includegraphics[scale=0.6]{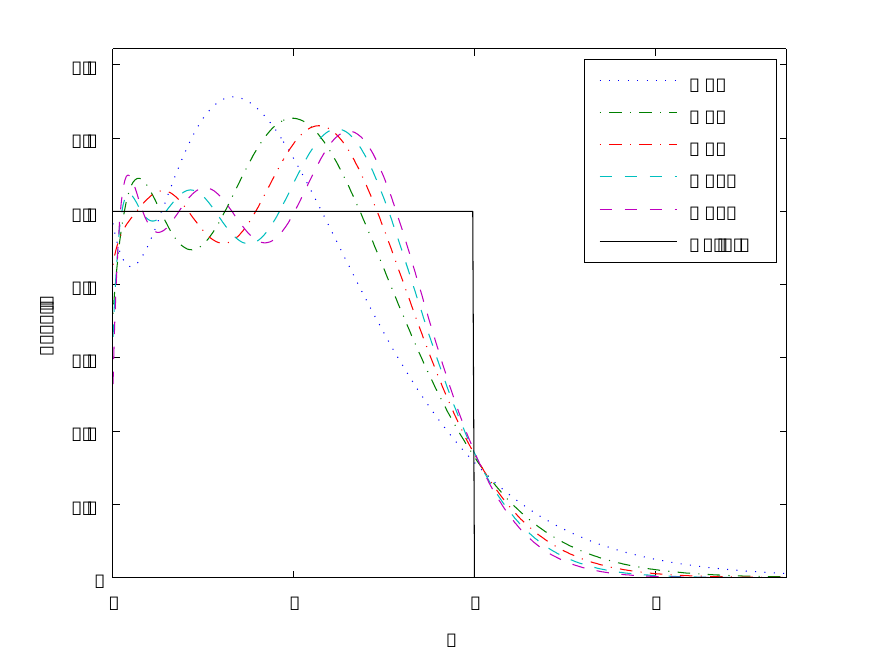} \\
(iii) Lognormal$(0,0.5)$ & (iv) Uniform$(0,2)$ 
\end{tabular}
\end{minipage}
\caption{Regular fit: densities of the target distribution and fitted PH-distributions for $m=3,6,9,12,15$.} \label{fitted_plot}
\end{center}
\end{figure}

\begin{figure}[htbp]
\begin{center}
\begin{minipage}{1.0\textwidth}
\centering
\begin{tabular}{cc}
\includegraphics[scale=0.6]{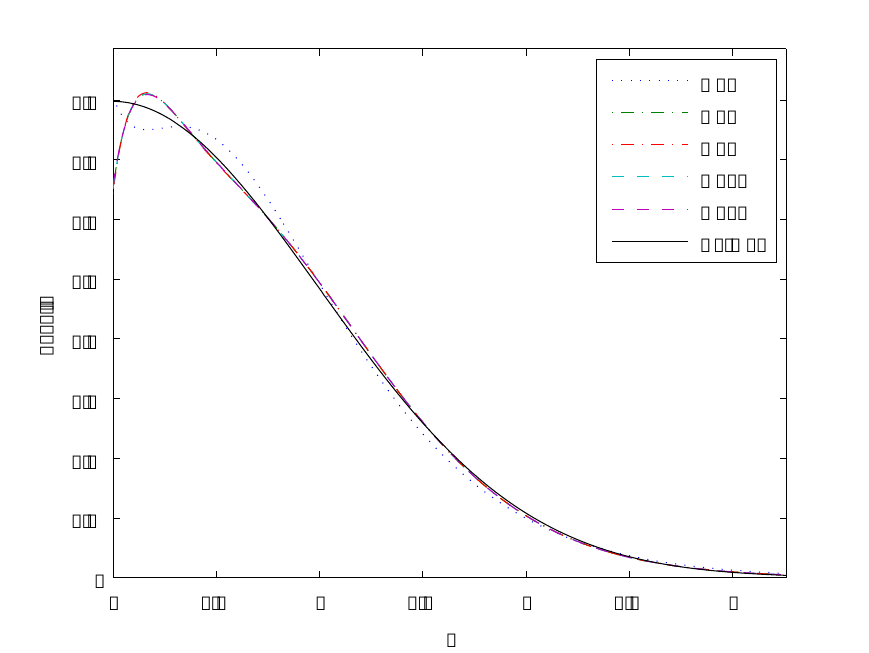}  & \includegraphics[scale=0.6]{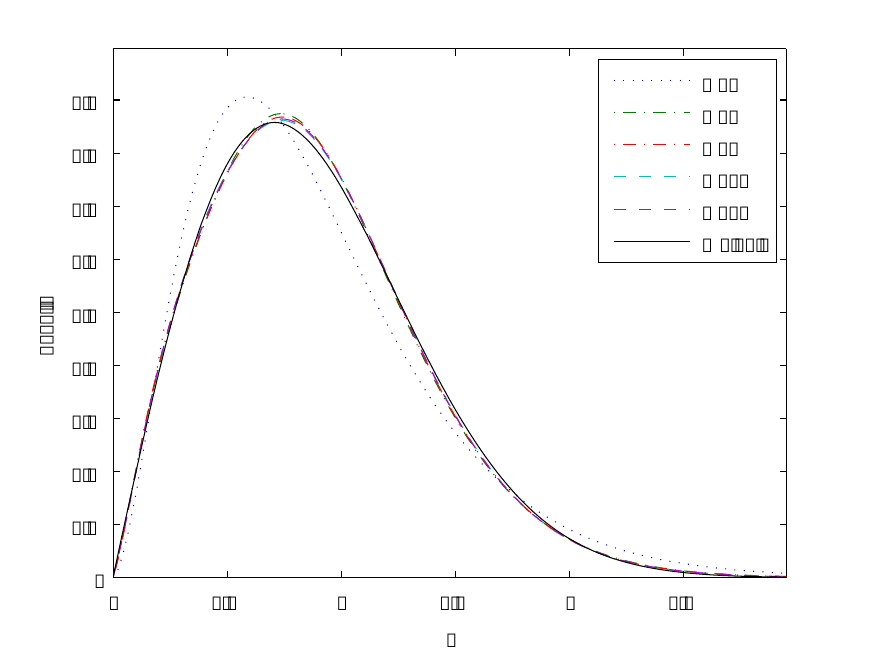}  \\
(i) Normal$(0,1)$ &  (ii) Weibull$(2,1)$ \\
\includegraphics[scale=0.6]{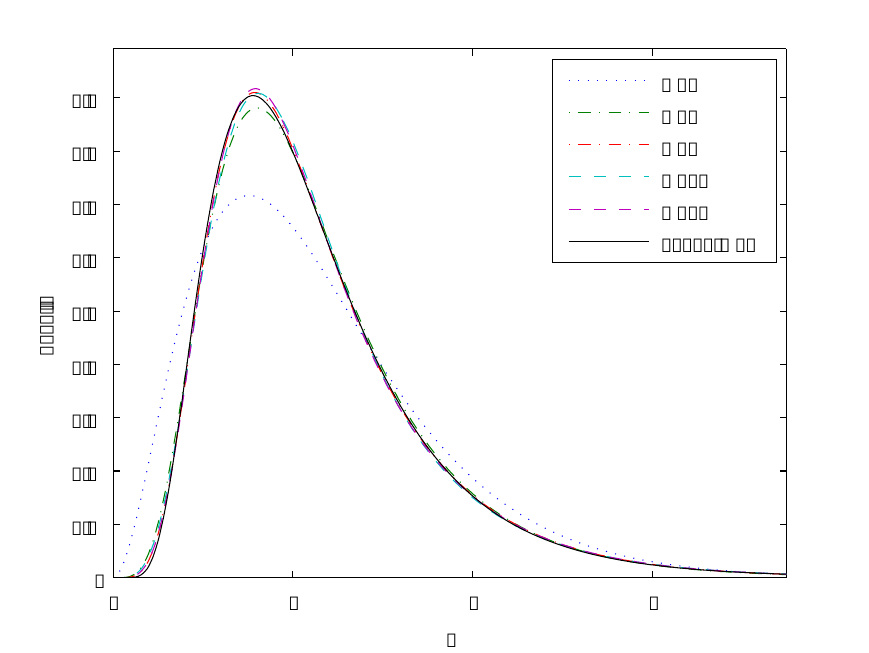}  & \includegraphics[scale=0.6]{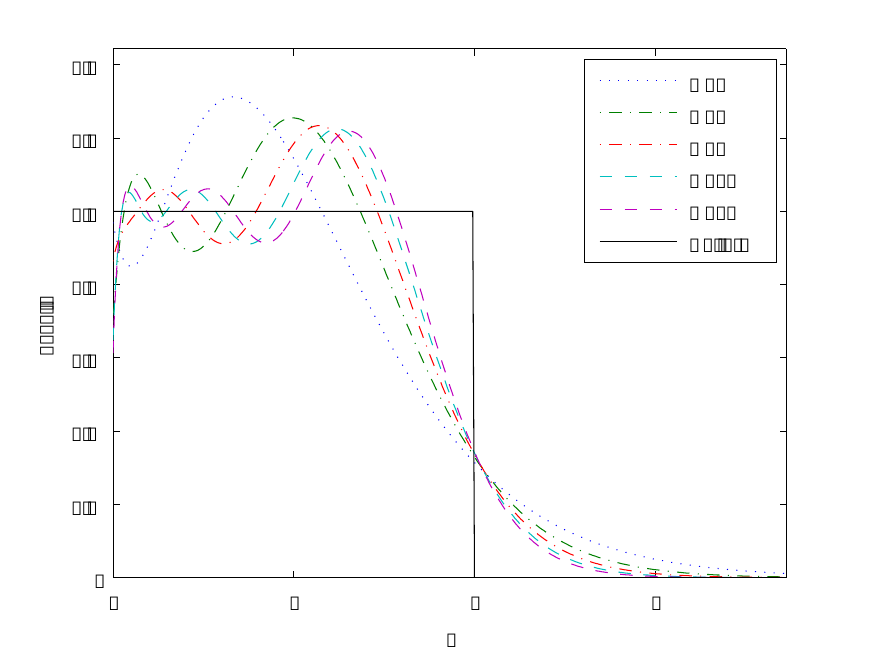} \\
(iii) Lognormal$(0,0.5)$ & (iv) Uniform$(0,2)$ 
\end{tabular}
\end{minipage}
\caption{Coxian fit: densities of the target distribution and fitted Coxian distributions for $m=3,6,9,12,15$.} \label{fitted_plot_coxian}
\end{center}
\end{figure}

\subsection{Computation of scale functions}  \label{numerics_EM}
%

We consider approximating the scale function for a compound Poisson process with i.i.d.\  jumps (with or without a Brownian motion component).  For the jump distribution, we consider (i) (the absolute values of) normal with mean zero and standard deviation $\nu = 1$, (ii) Weibull with $\beta = 2$ and $\gamma = 1$, (iii) lognormal with $\alpha = 0$ and $\kappa = 0.5$ and (iv) uniform with $a=0$ and $b=2$, where the probability densities $h$ at $x$ are, respectively,
\begin{align*}
\frac 2 {\sqrt{2 \pi \nu^2}} \exp \left\{ - \frac {x^2} {2 \nu^2} \right\}, \quad \beta \gamma^\beta x^{\beta -1} \exp \left\{ - (\gamma x)^\beta\right\}, \quad \frac 1 {\kappa x \sqrt{2 \pi}} \exp \left\{ - \frac {(\log x -\alpha)^2} {2 \kappa^2} \right\}, \quad \frac 1 {b-a}.
\end{align*}
Notice (ii) and (iii) do not admit Laplace exponents expressed in analytical form; as we discussed in introduction, the PH fitting approach has advantage over Laplace inversion method when the Laplace exponents do not admit analytical forms.

For various values of $m$ and for each jump distribution (i)-(iv) above, we fit a general PH-distribution (\emph{regular fit}) and a Coxian distribution (\emph{Coxian fit}).  This can be done because by construction the EM-algorithm preserves the zeros in $(\bm \alpha, \bm T)$ and one only needs to choose the initial estimate $({\bm \alpha^{(0)}, \bm T^{(0)}})$ in the desired class.  Here we consider also the Coxian fit because, as we have discussed in Example \ref{example_coxian}, it is potentially as powerful as the regular fit, and more importantly the computation time required to compute the scale function is expected to be smaller.
  In order to guarantee that the output has converged sufficiently, the value of $\delta({\bm \alpha^{(k)}, \bm T^{(k)}}; h)$ is monitored every $10$ steps;  the algorithm stops as soon as the difference 
$\delta({\bm \alpha^{(k)}, \bm T^{(k)}}; h)- \delta({\bm \alpha^{(k-10)}, \bm T^{(k-10)}}; h)$ becomes less than $10^{-6}$.

In Table \ref{computation_time}, for both the regular and Coxian fit, we show the time required to compute the approximation $(m,  \bm \alpha_m, \bm T_m)$ for (i)-(iv) and $m=3,6,9,12,15$ along with the required number of iterations. Because the initial input  $(\bm \alpha^{(0)},  \bm T^{(0)})$ is randomly chosen, the required run-time is random and hence does not necessarily increase in $m$.  However, overall the computation time tends to increase rapidly as $m$ increases, while the number of iterations does not.  This implies that the time required \emph{for each EM-step} increases rapidly in $m$, and hence it is not possible to make $m$ arbitrarily large.   Regarding the comparison between the regular and Coxian fit, we see that the latter is not necessarily faster than the former; the exception is the uniform case where the Coxian fit terminates quickly even for large $m$.

 The probability density functions of the target and the fitted PH-distributions are plotted in Figures \ref{fitted_plot} and  \ref{fitted_plot_coxian} for the regular and Coxian fit, respectively; see the appendix for the fitted PH-distributions for $m=3,9$.   For all cases,  no significant difference is  observed between the regular and Coxian fit in view of Figures \ref{fitted_plot} and  \ref{fitted_plot_coxian}. For the Weibull and lognormal cases, we see that the approximation gets more accurate as $m$ increases.  For the normal case, although the fitting is already reasonably accurate when $m=3$, the fitting in the neighborhood of zero is not accurately done for larger values of $m$.   Regarding the uniform case, the PH-fitting is known to be very difficult. Indeed, as far as the density approximation is concerned, it is far less accurate compared with the other three cases.  We shall see below, however, that the approximation of scale functions is nonetheless accurate.

We consider the \lev processes $X^{(\textrm{normal})}$, $X^{(\textrm{weibull})}$, $X^{(\textrm{log})}$ and $X^{(\textrm{unif})}$ in the form (\ref{levy_canonical}) where the distribution of $Z$ is given by  (i)-(iv), respectively, with common parameters $\mu = 5$ and  $q = 0.05$ and
\begin{enumerate}
\item[(a)] $\sigma = 1$ and $\lambda = 5$, 
\item[(b)]  $\sigma = 0$ and $\lambda = 5$,
\item[(c)]  $\sigma = 1$ and $\lambda = 10$. 
\end{enumerate}
 Using the fitted PH-distributions computed under the regular/Coxian fit, we construct PH \lev processes $\widetilde{X}^{(\textrm{normal})}$, $\widetilde{X}^{(\textrm{weibull})}$, $\widetilde{X}^{(\textrm{log})}$ and $\widetilde{X}^{(\textrm{unif})}$ for $m=3,9,15$, and see how their scale functions can be used as approximations.  The computation of the scale function amounts to computing the elements in $\mathcal{I}_q$ and $\mathcal{J}_q$, or the roots and poles of $\psi(\cdot) = q$.  Here we use the built-in MATLAB functions solve() and eig() to compute the former and the latter, respectively; these values for $m = 3,9$ are given in the appendix.  In all our numerical results given below, the obtained roots in $\mathcal{I}_q$ and the positive root $\Phi(q)$ are all distinct (and hence  the scale function is given by \eqref{scale_function_distinct}).  This is consistent with Remark \ref{remark_multiplicity_unlikely}.   Table \ref{computation_time_scale_function_ruin} shows the time required to compute the coefficients of the scale function for (a)-(c)  for both the regular and Coxian fit cases.  As is expected, the Coxian fit case runs much faster due to the sparsity of the PH-generator.  Indeed, while the computation time for the regular fit case increases nonlinearly  in $m$, it increases approximately linearly for the Coxian case.  


\subsection{Approximation of scale functions and derivatives}  We now evaluate the accuracy of the fitted scale functions as approximation tools. In our first experiment, we use the identity as in \eqref{laplace_in_terms_of_z}
\begin{align}
\E^x \left[ e^{-q \tau_b^+} 1_{\{ \tau_0^- > \tau_b^+, \,  \tau_b^+ < \infty \}}\right] = \frac {W^{(q)}(x)} {W^{(q)}(b)}, \label{criteria1}
\end{align}
and the well-known fluctuation identities of the \emph{reflected process}:
\begin{align}
\E^x \Big[ \int_0^{\nu^a} e^{-q t} \diff L_t^a \Big] =  \frac {W^{(q)}(x)}  {W^{(q)'}(a)}, \quad 0 \leq x \leq a. \label{criteria2}
\end{align}
Here $L_t^a := \sup_{0 \leq s \leq t} (X_s - a) \vee 0$, $t \geq 0$, and $\nu^a := \inf \{ t > 0: U_t^{a} < 0 \}$ is the time of ruin of the reflected process $U_t^a := X_t - L_t^a$.  This is an important quantity of interest in the field of insurance; $X$ is seen as the surplus of an insurance company and $L_t^a$ the cumulative amount of dividends under the barrier strategy with barrier level $a$.  We refer the reader to, among others,  \cite{Avram_et_al_2007} and \cite{Loeffen_2008} regarding the insurance dividend problem for a spectrally negative \lev process.

We compute the right-hand sides of \eqref{criteria1} and \eqref{criteria2} for $\widetilde{X}$ explicitly via \eqref{scale_function_distinct} and approximate the left-hand sides for $X$ via Monte Carlo simulation based on $100,000$ sample paths.  For the simulated results, Brownian motions are approximated by random walks with time step $\Delta t = T/ 100$ for each interarrival time $T$ between jumps.  
We consider starting points $x = 1,\ldots, 4$ with common parameters $a=b = 5$.  Tables \ref{scale_results_general} and  \ref{results_derivative_general} give the results for (a)  and Tables \ref{scale_results_general_no_gaussian} and  \ref{results_derivative_general_no_gaussian} give the results for (b).  In both figures, we show the computation time required for the simulated results; that of the PH-fitting is omitted because these can be computed instantaneously (once the parameters of the scale function are computed).

From these tables, we see that these expectations for $X$ are approximated very precisely, and can infer that the scale functions of $X$ and their derivatives are approximated efficiently by those of $\widetilde{X}$.  We also see that the performance between the regular and Coxian fit is almost the same in all cases; hence in view of the computation time discussed above, the Coxian fit is indeed a powerful alternative to the regular fit. 
The approximation overall tends to improve in $m$ (which may not be clear for the normal case).  Except for the uniform case, the differences between the cases $m=9$ and $m=15$ are negligible and hence we can infer that increasing the value of $m$ further would not have a significant improvement in the approximation.  For the uniform case, on the other hand, we expect  that we can improve it by choosing $m$ higher.  However, in spite of the performance of the density approximation for the uniform case as in Figures \ref{fitted_plot} and  \ref{fitted_plot_coxian}, the approximation results for the uniform case are surprisingly accurate.


\begin{table}[ht] 
\begin{tabular}{c}
\centering 
\begin{tabular}{|r|rr|rr|rr|rr|} 
\hline
& \multicolumn{2}{c|}{$m=3$}& \multicolumn{2}{c|}{$m=9$} & \multicolumn{2}{c|}{$m=15$} & \multicolumn{2}{c|}{simulation}\\ 
\hline 
$x$ & Regular & Coxian & Regular & Coxian & Regular & Coxian & mean (95\% CI) & time \\
\hline
1&0.41128& 0.41127 & 0.40906& 0.40902 & 0.40900& 0.40904 &    0.41226 (0.40922, 0.41530) & 329.81\\
2&0.60625& 0.60625 & 0.60565& 0.60563 & 0.60560& 0.60563 &  0.60712 (0.60434, 0.60991)& 373.98\\
3&0.76295& 0.76295 & 0.76313& 0.76313 & 0.76312& 0.76313 &  0.76651 (0.76429, 0.76873) &  333.79\\
4&0.89141 & 0.89141 & 0.89167& 0.89167 & 0.89166 & 0.89167 &  0.89154 (0.88982,  0.89327) & 209.88\\
\hline
\end{tabular}   
\end{tabular} \\
(i) Normal$(0,1)$ \\
%
%
\begin{tabular}{c}
\centering 
\begin{tabular}{|r|rr|rr|rr|rr|} 
\hline
& \multicolumn{2}{c|}{$m=3$}& \multicolumn{2}{c|}{$m=9$} & \multicolumn{2}{c|}{$m=15$} & \multicolumn{2}{c|}{simulation}\\ 
\hline 
$x$ & Regular & Coxian & Regular & Coxian & Regular & Coxian & mean (95\% CI) & time \\
\hline
1&0.33262& 0.33262 & 0.32745& 0.32742 & 0.32722& 0.32735 &   0.32862
 (0.32617,  0.33107) & 468.46\\
2&0.52664& 0.52664 & 0.52509& 0.52508 & 0.52501 & 0.52506 &   0.52326 (0.52087, 0.52566)&  566.18\\
3&0.69920& 0.69920 & 0.69906& 0.69906 & 0.69904& 0.69905 &  0.69759 (0.69511,  0.70008) &  512.66\\
4&0.85546 & 0.85546 &0.85572 & 0.85572& 0.85572 & 0.85572 &   0.85457 (0.85246, 0.85669) &  333.86\\
\hline
\end{tabular}   
\end{tabular} \\
(ii) Weibull$(2,1)$ \\
\begin{tabular}{c}
\centering 
\begin{tabular}{|r|rr|rr|rr|rr|} 
\hline
& \multicolumn{2}{c|}{$m=3$}& \multicolumn{2}{c|}{$m=9$} & \multicolumn{2}{c|}{$m=15$} & \multicolumn{2}{c|}{simulation}\\ 
\hline 
$x$ & Regular & Coxian & Regular & Coxian & Regular & Coxian & mean (95\% CI) & time \\
\hline
1&0.18763& 0.18763 & 0.18258 & 0.18226 & 0.18262& 0.18248 &   0.18389 (0.18180, 0.18597) &  237.75\\
2&0.33763& 0.33766 & 0.33267&  0.33222 & 0.33272 & 0.33248&  0.33231 (0.32932, 0.33531)&  313.26\\
3&0.51826&  0.51826 & 0.51335 &0.51297 & 0.51340& 0.51319  &  0.51447 (0.51166, 0.51727)  & 313.31\\
4&0.73603 &  0.73603 &0.73259 & 0.73234& 0.73262 & 0.73249 &  0.73018 (0.72764, 0.73273)  &  228.31\\
\hline
\end{tabular}   
\end{tabular} \\
(iii) Lognormal$(0,0.5)$ \\
\begin{tabular}{c}
\centering 
\begin{tabular}{|r|rr|rr|rr|rr|} 
\hline
& \multicolumn{2}{c|}{$m=3$}& \multicolumn{2}{c|}{$m=9$} & \multicolumn{2}{c|}{$m=15$} & \multicolumn{2}{c|}{simulation}\\ 
\hline 
$x$ & Regular & Coxian & Regular & Coxian & Regular & Coxian & mean (95\% CI) & time \\
\hline
1&0.26009& 0.26009  & 0.24810 & 0.24810 & 0.24619& 0.24620 &   0.24564 (0.24279, 0.24850) & 259.51\\
2&0.43173& 0.43173    &0.42313&  0.42313 & 0.42192   & 0.42192&   0.42320 (0.42008, 0.42631) & 334.44\\
3&0.61132&  0.61132 & 0.60514&0.60514 & 0.60416& 0.60417  &  0.60567 (0.60299, 0.60834)  &  321.42\\
4&0.79996 &  0.79996 &0.79660 & 0.79660&0.79609 & 0.79609 &   0.79465 (0.79224, 0.79705) & 227.91\\
\hline
\end{tabular}   
\end{tabular} \\
(iv) Uniform$(0,2)$
\caption{Computation of $\E^x [ e^{-q \tau_b^+} 1_{\{ \tau_0^- > \tau_b^+, \, \tau_b^+ < \infty \}}]$ via scale function and simulation for (a) $\sigma = 1$ and $\lambda = 5$.} \label{scale_results_general}
\end{table}

\begin{table} 
\begin{tabular}{c}
\centering 
\begin{tabular}{|r|rr|rr|rr|rr|} 
\hline
& \multicolumn{2}{c|}{$m=3$}& \multicolumn{2}{c|}{$m=9$} & \multicolumn{2}{c|}{$m=15$} & \multicolumn{2}{c|}{simulation}\\ 
\hline 
$x$ & Regular & Coxian & Regular & Coxian & Regular & Coxian & mean (95\% CI) & time \\
\hline
1&0.44287& 0.44286 & 0.44071& 0.44066 & 0.44062& 0.44067 &    0.44247 (0.43948, 0.44546) & 396.20\\
2&0.63249& 0.63249 & 0.63235& 0.63234 & 0.63231& 0.63233 &  0.63625 (0.63330, 0.63919)& 433.87\\
3&0.78076& 0.78076 & 0.78127& 0.78128 & 0.78126& 0.78127 &  0.78556 (0.78323, 0.78789) & 376.56\\
4&0.90013 & 0.90013 &0.90056& 0.90056 & 0.90056 & 0.90056 &  0.90249 (0.90087, 0.90411) & 234.00\\
\hline
\end{tabular}   
\end{tabular} \\
(i) Normal$(0,1)$\\
\begin{tabular}{c}
\centering 
\begin{tabular}{|r|rr|rr|rr|rr|} 
\hline
& \multicolumn{2}{c|}{$m=3$}& \multicolumn{2}{c|}{$m=9$} & \multicolumn{2}{c|}{$m=15$} & \multicolumn{2}{c|}{simulation}\\ 
\hline 
$x$ & Regular & Coxian & Regular & Coxian & Regular & Coxian & mean (95\% CI) & time \\
\hline
1&0.35433& 0.35433 & 0.34982& 0.34980 & 0.34958& 0.34972 &   0.34905 (0.34610,  0.35201) & 552.43\\
2& 0.54527& 0.54527 & 0.54455& 0.54454 & 0.54449& 0.54453 &   0.54324 (0.54023, 0.54625)&  664.64\\
3&0.71240& 0.71240 & 0.71298& 0.71298 & 0.71297 & 0.71297 &  0.71205 (0.70965, 0.71445) & 586.03\\
4&0.86221 & 0.86221 &0.86286& 0.86286 & 0.86287 & 0.86286 &  0.86239 (0.86060, 0.86418) & 367.62\\
\hline
\end{tabular}   
\end{tabular} \\
(ii) Weibull$(2,1)$ \\
\begin{tabular}{c}
\centering 
\begin{tabular}{|r|rr|rr|rr|rr|} 
\hline
& \multicolumn{2}{c|}{$m=3$}& \multicolumn{2}{c|}{$m=9$} & \multicolumn{2}{c|}{$m=15$} & \multicolumn{2}{c|}{simulation}\\ 
\hline 
$x$ & Regular & Coxian & Regular & Coxian & Regular & Coxian & mean (95\% CI) & time \\
\hline
1&0.18117& 0.18117  & 0.17583& 0.17543 & 0.17586& 0.17568 &   0.17619 (0.17415, 0.17824) & 282.18\\
2& 0.32577& 0.32577 & 0.31968& 0.31920  & 0.31974 & 0.31948 &  0.31931
 (0.31640, 0.32221)&  350.95\\
3&0.50410  & 0.50410 & 0.49796& 0.49752  & 0.49801  & 0.49777 &  0.49706 (0.49433,  0.49979) & 351.18\\
4&0.72499 & 0.72498&0.72056& 0.72027 & 0.72059 & 0.72044 &  0.72331 (0.72065, 0.72598) &  251.99\\
\hline
\end{tabular}   
\end{tabular} \\
(iii) Lognormal$(0,0.5)$ \\
\begin{tabular}{c}
\centering 
\begin{tabular}{|r|rr|rr|rr|rr|} 
\hline
& \multicolumn{2}{c|}{$m=3$}& \multicolumn{2}{c|}{$m=9$} & \multicolumn{2}{c|}{$m=15$} & \multicolumn{2}{c|}{simulation}\\ 
\hline 
$x$ & Regular & Coxian & Regular & Coxian & Regular & Coxian & mean (95\% CI) & time \\
\hline
1&0.26487& 0.26487 & 0.25203 & 0.25202 & 0.24969& 0.24969  &    0.24927 (0.24666, 0.25188) & 319.60\\
2&0.43371& 0.43371 & 0.42447&  0.42447 & 0.42312 & 0.42313 &  0.42229 (0.41947, 0.42512)&  417.88\\
3&0.61112 &  0.61112 & 0.60435 &0.60435 & 0.60329& 0.60329  &  0.60631 (0.60317, 0.60945)  &  388.21\\
4&0.79898 &  0.79898 &0.79523 &  0.79523& 0.79465 & 0.79465 &  0.79718 (0.79490, 0.79946) &  264.30\\
\hline
\end{tabular}   
\end{tabular} \\
(iv) Uniform$(0,2)$ \\
\caption{Computation of $\E^x [ e^{-q \tau_b^+} 1_{\{ \tau_0^- > \tau_b^+, \, \tau_b^+ < \infty \}}]$ via scale function and simulation for (b) $\sigma = 0$ and $\lambda = 5$. }  \label{scale_results_general_no_gaussian}
\end{table}

\begin{table}[ht] 
\begin{tabular}{c}
\centering 
\begin{tabular}{|r|rr|rr|rr|rr|} 
\hline
& \multicolumn{2}{c|}{$m=3$}& \multicolumn{2}{c|}{$m=9$} & \multicolumn{2}{c|}{$m=15$} & \multicolumn{2}{c|}{simulation}\\ 
\hline 
$x$ & Regular & Coxian & Regular & Coxian & Regular & Coxian & mean (95\% CI) & time \\
\hline
1&4.07774& 4.07765 & 4.07050& 4.07033 & 4.06985& 4.07029 &   4.17448 (4.13553,  4.21343) & 172.14\\
2&6.01078& 6.01082 & 6.02667& 6.02682 & 6.02617& 6.02653 &   6.15991 (6.11411, 6.20571)& 235.84\\
3&7.56442& 7.56452 & 7.59374& 7.59416 & 7.59355& 7.59379 &   7.77769 (7.73176, 7.82362) & 275.15\\
4&8.83798 & 8.83812 & 8.87277& 8.87330 & 8.87269 & 8.87289 &  9.01534 (8.97360, 9.05709) &  295.44\\
\hline
\end{tabular}   
\end{tabular} \\
(i) Normal$(0,1)$ \\
\begin{tabular}{c}
\centering 
\begin{tabular}{|r|rr|rr|rr|rr|} 
\hline
& \multicolumn{2}{c|}{$m=3$}& \multicolumn{2}{c|}{$m=9$} & \multicolumn{2}{c|}{$m=15$} & \multicolumn{2}{c|}{simulation}\\ 
\hline 
$x$ & Regular & Coxian & Regular & Coxian & Regular & Coxian & mean (95\% CI) & time \\
\hline
1&2.37597& 2.37596 & 2.34729& 2.34712 & 2.34578& 2.34668 &   2.34652 (2.31636, 2.37667) & 638.44\\
2&3.76185 & 3.76187 & 3.76411& 3.76407 & 3.76368& 3.76395 &  3.77443 (3.74053, 3.80833) & 940.63\\
3&4.99446& 4.99448 & 5.01123& 7.59416 & 5.01125& 5.01124 &  5.00423 (4.96857, 5.03990) & 1128.37\\
4&6.11062 & 6.11064 & 6.13423& 8.87330 & 6.13444 & 6.13430 &  6.14054 (6.10623, 6.17486) &  1247.02\\
\hline
\end{tabular}   
\end{tabular} \\
(ii) Weibull$(2,1)$ \\
\begin{tabular}{c}
\centering 
\begin{tabular}{|r|rr|rr|rr|rr|} 
\hline
& \multicolumn{2}{c|}{$m=3$}& \multicolumn{2}{c|}{$m=9$} & \multicolumn{2}{c|}{$m=15$} & \multicolumn{2}{c|}{simulation}\\ 
\hline 
$x$ & Regular & Coxian & Regular & Coxian & Regular & Coxian & mean (95\% CI) & time \\
\hline
1&0.64568& 0.64568 & 0.61838& 0.61663  & 0.61859& 0.61776 &  0.61080 (0.60015,  0.62144) & 53.24\\
2&1.16200  & 1.16200  & 1.12670& 1.12399  & 1.12701 & 1.12557 &  1.12263 (1.10908, 1.13617) & 76.16\\
3&1.78348 & 1.78348 & 1.73866& 1.73550 & 1.73903& 1.73733 &  1.73700 (1.71959, 1.75440) & 93.41\\
4&2.53289 & 2.53289 & 2.48118& 2.47769 & 2.48156 & 2.47970 &  2.49715 (2.47926, 2.51504) & 105.02 \\
\hline
\end{tabular}   
\end{tabular} \\
(iii) Lognormal$(0,0.5)$ \\
\begin{tabular}{c}
\centering 
\begin{tabular}{|r|rr|rr|rr|rr|} 
\hline
& \multicolumn{2}{c|}{$m=3$}& \multicolumn{2}{c|}{$m=9$} & \multicolumn{2}{c|}{$m=15$} & \multicolumn{2}{c|}{simulation}\\ 
\hline 
$x$ & Regular & Coxian & Regular & Coxian & Regular & Coxian & mean (95\% CI) & time \\
\hline
1&1.25907& 1.25909 & 1.18000 & 1.17998 & 1.16775& 1.16778 &  1.18704 (1.16908, 1.20501)& 86.48\\
2&2.08997  & 2.08999  & 2.01243&2.01243  & 2.00127 & 2.00130 &  1.98021 (1.95857, 2.00185) & 125.20\\
3&2.95935 & 2.95938 & 2.87810& 2.87810  & 2.86571& 2.86574 &   2.84949 (2.82471, 2.87427) & 153.69\\
4&3.87252 & 3.87255 & 3.78872& 3.78872 & 3.77606 & 3.77609 &   3.72817 (3.70147, 3.75487) & 171.65\\
\hline
\end{tabular}   
\end{tabular} \\
(iv) Uniform$(0,2)$ \\
\caption{Computation of $\E^x \big[ \int_0^{\nu^a} e^{-q t} \diff L_t^a \big] $ via scale function and simulation for (a) $\sigma = 1$ and $\lambda = 5$.  } \label{results_derivative_general}
\end{table}

\begin{table} 
\begin{tabular}{c}
\centering 
\begin{tabular}{|r|rr|rr|rr|rr|} 
\hline
& \multicolumn{2}{c|}{$m=3$}& \multicolumn{2}{c|}{$m=9$} & \multicolumn{2}{c|}{$m=15$} & \multicolumn{2}{c|}{simulation}\\ 
\hline 
$x$ & Regular & Coxian & Regular & Coxian & Regular & Coxian & mean (95\% CI) & time \\
\hline
1&4.787254& 4.78722 & 4.79092& 4.79083 & 4.79009 & 4.79062 &   4.88319 (4.83846, 4.92793) & 252.25\\
2&6.83693& 6.83707 & 6.87422& 6.87474 & 6.87394& 6.87425 &  7.01037 (6.95899, 7.06174)& 335.78\\
3&8.43963& 8.43984 & 8.49315& 8.49401 & 8.49325& 8.49341 &   8.67669 (8.63073,  8.72265) & 393.36\\
4&9.73000 & 9.73025 & 9.78990 & 9.79088 & 9.79014 & 9.79024 &  9.92484 (9.87496, 9.97473) & 417.80\\
\hline
\end{tabular}   
\end{tabular} \\
(i) Normal$(0,1)$ \\
\begin{tabular}{c}
\centering 
\begin{tabular}{|r|rr|rr|rr|rr|} 
\hline
& \multicolumn{2}{c|}{$m=3$}& \multicolumn{2}{c|}{$m=9$} & \multicolumn{2}{c|}{$m=15$} & \multicolumn{2}{c|}{simulation}\\ 
\hline 
$x$ & Regular & Coxian & Regular & Coxian & Regular & Coxian & mean (95\% CI) & time \\
\hline
1&2.65896& 2.65897 & 2.64269& 2.64253 & 2.64121 & 2.64209 &   2.63314 (2.60486, 2.66142) &  855.55\\
2&4.09177& 4.09180 & 4.11376& 4.11377 & 4.11381& 4.11379  &  4.08878 (4.05548, 4.12209)& 1210.17\\
3&5.34596& 5.34600 & 5.38611& 5.38618 & 5.38675& 5.38637 &  5.39638 (5.35693, 5.43583) & 1458.89\\
4&6.47014 & 6.47018 & 6.51839& 6.51850 & 6.51931 & 6.51877 & 6.53982 (6.49354, 6.58611) &  1606.43\\
\hline
\end{tabular}   
\end{tabular} \\
(ii) Weibull$(2,1)$ \\
\begin{tabular}{c}
\centering 
\begin{tabular}{|r|rr|rr|rr|rr|} 
\hline
& \multicolumn{2}{c|}{$m=3$}& \multicolumn{2}{c|}{$m=9$} & \multicolumn{2}{c|}{$m=15$} & \multicolumn{2}{c|}{simulation}\\ 
\hline 
$x$ & Regular & Coxian & Regular & Coxian & Regular & Coxian & mean (95\% CI) & time \\
\hline
1&0.59077& 0.59077 & 0.56206 & 0.56008 & 0.56221  & 0.56131 &   0.55884 (0.54929, 0.56840) & 59.27\\
2&1.06227& 1.06227  & 1.02191&  1.01910 & 1.02220 & 1.02073  &  1.02341 (1.00922, 1.03761)&  84.33\\
3&1.64378& 1.64378 & 1.59179& 1.58843 & 1.59212& 1.59036  &  1.59522 (1.57935, 1.61108)& 104.18\\
4&2.36403 & 2.36402 & 2.30339&  2.29962 & 2.30373 & 2.30177 & 2.30114 (2.28562, 2.28562) & 120.21\\
\hline
\end{tabular}   
\end{tabular} \\
(iii) Lognormal$(0,0.5)$ \\
\begin{tabular}{c}
\centering 
\begin{tabular}{|r|rr|rr|rr|rr|} 
\hline
& \multicolumn{2}{c|}{$m=3$}& \multicolumn{2}{c|}{$m=9$} & \multicolumn{2}{c|}{$m=15$} & \multicolumn{2}{c|}{simulation}\\ 
\hline 
$x$ & Regular & Coxian & Regular & Coxian & Regular & Coxian & mean (95\% CI) & time \\
\hline
1&1.27003 & 1.27005 & 1.18464& 1.18462  & 1.17009& 1.17010 &  1.18704 (1.16908, 1.20501)& 86.48\\
2&2.07961  & 2.07964  & 1.99519& 1.99519   & 1.98281 & 1.98285 &  1.98021 (1.95857, 2.00185) & 125.20\\
3&2.93024& 2.93027 & 2.84069& 2.84070 & 2.82708& 2.82712&  2.84949 (2.82471, 2.87427) & 153.69\\
4&3.83102 & 3.83105 & 3.73795& 3.73796 & 3.72385 & 3.72389 &  3.72817 (3.70147, 3.75487) & 171.65 \\
\hline
\end{tabular}   
\end{tabular} \\
(iv) Uniform$(0,2)$ \\
\caption{Computation of $\E^x \big[ \int_0^{\nu^a} e^{-q t} \diff L_t^a \big] $ via scale function and simulation for (b) $\sigma = 0$ and $\lambda = 5$.} \label{results_derivative_general_no_gaussian}
\end{table}

\subsection{Approximation of resolvent measures} In our second experiment, we shall evaluate the approximation of some functionals integrated with respect to the resolvent measure \eqref{eq_resolvent}.  Here we consider only for the case $\sigma > 0$ because we have observed above that existence/non-existence of the Brownian motion component does not have any noticeable impact in the approximation accuracy. 

Here we give three examples for $f$:
\begin{enumerate}
\item[(L)] \emph{linear function }: $f^{(lin)}(y) :=y$;
\item[(E)] \emph{capped exponential function}: $f^{(exp)}(y) := e^{y \wedge B}$ for some $B \geq 0$.
\item[(S)] \emph{simple function}: $f^{(sim)}(y) := \sum_{n \geq 1} f^{(n)} 1_{I_n} (y)$ for some constants $ f^{(n)}$  and subdivisions $I_n := (l_n, l_{n+1}]$ of $[0,\infty)$.
\end{enumerate}
In particular, we choose $B=1$ for (E) and  $I_1 = (0, 3)$, $I_2 = [3,\infty)$, $f^{(1)}=-10$ and  $f^{(2)}=10$ for (S).

We consider  \eqref{eq_resolvent} with $A = 0$, or the function $U(x) := \E^{x}\big[ \int_0^{\tau_{0}^-} e^{-qt} f(X_t) \diff t \big] =
W^{(q)} (x) \Psi_f (0)-
\Theta_f(x; 0)$.  Assuming the roots in $\mathcal{I}_q$ are distinct, straightforward integration gives
\begin{align*}
\Psi_{f^{(lin)}}(0)  &=   \frac 1 {\lapinv^2}, \\
\Psi_{f^{(exp)}}(0)
&= \frac {e^B} {\lapinv} e^{-\lapinv B} + \frac 1 {\Phi^{(1)}(q)} \Big[ 1- e^{-\Phi^{(1)}(q) B}  \Big], \\
\Psi_{f^{(sim)}}(0)  &= \frac 1 {\lapinv} \sum_{n \geq 1} f^{(n)} \Big[ e^{-\lapinv l_n} - e^{-\lapinv l_{n+1}}\Big], \\
\end{align*}
and
\begin{align*}
\Theta_{f^{(lin)}}(x;0) &=  \sum_{i \in \mathcal{I}_q} B_{i,q} [e^{\lapinv x} \gamma^{(-\Phi(q))}( 0, x) - e^{-\xi_{i,q}x}\gamma^{(\xi_{i,q})}( 0, x)], \\
 \Theta_{f^{(exp)}}(x;0)  
&= e^B w^{(0)}_{x}(B \wedge x, x) + e^x w^{(1)}_{x}(0, B \wedge x), \\
\Theta_{f^{(sim)}}(x;0) &= \sum_{n \geq 1} f^{(n)} w^{(0)}_{x}( l_n, l_{n+1}).
\end{align*}
Here we define, for any $0 \leq s < t$ and $k \geq 0$,
\begin{align*}
w_{x}^{(k)}(s, t) &:=  \frac 1 {\Phi^{(k)}(q) \psi'(\Phi(q))} \left( e^{\Phi^{(k)}(q)(x-s)_+} - e^{\Phi^{(k)}(q)(x-t)_+} \right) + \sum_{i  \in \mathcal{I}_q} B_{i,q} \Big[ \frac 1 {\xi_{i,q}^{(k)}} \Big( e^{-\xi_{i,q}^{(k)}(x-s)_+} - e^{-\xi_{i,q}^{(k)}(x-t)_+}  \Big) \Big],
\end{align*}
where $\Phi^{(k)}(q) := \Phi(q) - k$ and  $\xi^{(k)}_{i,q} := \xi_{i,q} + k$.
Also for  any $s < t$ and $a \in \R$, we define $\gamma^{(a)}(s, t) := \int_s^t y e^{a y} \diff y$.



We again evaluate the approximation by comparing them with simulated results.  In order to expedite the simulation, we use the parameter set (c) to decrease the value of $\tau_0^-$; the other parameters for the \lev processes are the same  as above.

Tables \ref{results_lin}, \ref{results_exp} and \ref{results_sim} show the results for (L), (E) and (S), respectively, along with the simulation time.   Overall, the results are accurate for $m=9,15$ while they are not for $m=3$.  As is expected, the approximation for the simple case is more difficult than the other two cases due to its discontinuity $-10 = f^{(sim)}(3-) < f^{(sim)}(3+) = 10$. The approximation improves sharply as we increase the value of $m$ and clearly there is a room for improvement by choosing a higher value of $m$.    We again did not observe any non-negligible difference between the regular and Coxian fit.

\subsection{Summary}  We have learned that the PH-fitting approach is a practical tool for approximating a scale function at least when the process to be fitted has a finite \lev measure.  While the computation time increases rapidly in the number of phases $m$,  a moderate value of $m$ (around $10$) attains a reasonably accurate fit.  Surprisingly, it still works even for the uniform case, which has been known to be hard to fit PH-distributions.  One important lesson we have learned here is that the Coxian fit is faster and as accurate as the regular fit.


\begin{table}[ht] 
\begin{tabular}{c}
\centering 
\begin{tabular}{|c|rr|rr|rr|cr|} 
\hline
& \multicolumn{2}{c|}{$m=3$}& \multicolumn{2}{c|}{$m=9$} & \multicolumn{2}{c|}{$m=15$} & \multicolumn{2}{c|}{simulation}\\ 
\hline 
$x$ & \multicolumn{1}{c}{Regular} & \multicolumn{1}{c|}{Coxian} & \multicolumn{1}{c}{Regular} & \multicolumn{1}{c|}{Coxian} & \multicolumn{1}{c}{Regular} & \multicolumn{1}{c|}{Coxian} & \multicolumn{1}{c}{mean (95\% CI)} & \multicolumn{1}{c|}{time} \\
\hline
1&0.95352& 0.95345 & 0.94042 &  0.94023  & 0.94035 & 0.94036   &    0.95411 (0.93881, 0.96942) & 564.70\\
2&2.00009& 2.00000  & 1.97938 & 1.97903 & 1.97917& 1.97922&   2.00829 (1.98601, 2.03058)&  936.51\\
3&3.35143& 3.35130 & 3.32382&3.32336  & 3.32355& 3.32361 &  3.39018 (3.36389, 3.41646)  & 1309.76\\
4&5.00098 & 5.00083 & 4.96663& 4.96605 & 4.96628 & 4.96636 &  5.02722 (4.98984, 5.06460) & 1681.43\\
\hline
\end{tabular}   
\end{tabular} \\
(i) Normal$(0,1)$ \\
\begin{tabular}{c}
\centering 
\begin{tabular}{|c|rr|rr|rr|cr|} 
\hline
& \multicolumn{2}{c|}{$m=3$}& \multicolumn{2}{c|}{$m=9$} & \multicolumn{2}{c|}{$m=15$} & \multicolumn{2}{c|}{simulation}\\ 
\hline 
$x$ & \multicolumn{1}{c}{Regular} & \multicolumn{1}{c|}{Coxian} & \multicolumn{1}{c}{Regular} & \multicolumn{1}{c|}{Coxian} & \multicolumn{1}{c}{Regular} & \multicolumn{1}{c|}{Coxian} & \multicolumn{1}{c}{mean (95\% CI)} & \multicolumn{1}{c|}{time} \\
\hline
1&0.61395& 0.61393 & 0.58704  & 0.58696   & 0.58636 & 0.58676 &   0.58344 (0.57629,  0.59058)& 535.35\\
2&1.35687 &  1.35684  & 1.31552  &   1.31540  & 1.31447& 1.31509   &  1.31059 (1.29711, 1.32407)& 904.16\\
3& 2.34319  &  2.34315  &  2.28703 & 2.28687   & 2.28558 & 2.28644 &  2.28431 ( 2.26614,  2.30248) & 1267.34\\
4&3.56952 & 3.56947 & 3.49896& 3.49875 & 3.49712 & 3.49821 & 3.50308 (3.48293, 3.52322) &  1642.37\\
\hline
\end{tabular}   
\end{tabular} \\
(ii) Weibull$(2,1)$ \\
\begin{tabular}{c}
\centering 
\begin{tabular}{|c|rr|rr|rr|cr|} 
\hline
& \multicolumn{2}{c|}{$m=3$}& \multicolumn{2}{c|}{$m=9$} & \multicolumn{2}{c|}{$m=15$} & \multicolumn{2}{c|}{simulation}\\ 
\hline 
$x$ & \multicolumn{1}{c}{Regular} & \multicolumn{1}{c|}{Coxian} & \multicolumn{1}{c}{Regular} & \multicolumn{1}{c|}{Coxian} & \multicolumn{1}{c}{Regular} & \multicolumn{1}{c|}{Coxian} & \multicolumn{1}{c}{mean (95\% CI)} & \multicolumn{1}{c|}{time} \\
\hline
1&0.37773  & 0.37773 & 0.36101 & 0.36061 & 0.36101  & 0.36088   &    0.36592 (0.36209,  0.36976) & 283.95\\
2&0.82879 & 0.82879  & 0.80543&  0.80433 & 0.80544 & 0.80497   &  0.80776 (0.80158, 0.81395)& 465.72\\
3&1.43259 & 1.43259 & 1.40166 &  1.40004 & 1.40174 & 1.40098   &  1.39661 (1.38566, 1.40756)  &  642.24\\
4&2.18726 & 2.18726 &  2.14890&   2.14676 &  2.14904 &  2.14800 &  2.14715 ( 2.13563,  2.15866) & 821.78\\
\hline
\end{tabular}   
\end{tabular} \\
(iii) Lognormal$(0,0.5)$ \\
\begin{tabular}{c}
\centering 
\begin{tabular}{|c|rr|rr|rr|cr|} 
\hline
& \multicolumn{2}{c|}{$m=3$}& \multicolumn{2}{c|}{$m=9$} & \multicolumn{2}{c|}{$m=15$} & \multicolumn{2}{c|}{simulation}\\ 
\hline 
$x$ & \multicolumn{1}{c}{Regular} & \multicolumn{1}{c|}{Coxian} & \multicolumn{1}{c}{Regular} & \multicolumn{1}{c|}{Coxian} & \multicolumn{1}{c}{Regular} & \multicolumn{1}{c|}{Coxian} & \multicolumn{1}{c}{mean (95\% CI)} & \multicolumn{1}{c|}{time} \\
\hline
1&0.51113 & 0.51115 & 0.47752& 0.47753   & 0.47295& 0.47299 &   0.46909 (0.46325,  0.47494)& 339.06\\
2&1.10011  & 1.10014  & 1.04852& 1.04854   & 1.04145 & 1.04150 &   1.04423 (1.03445, 1.05401)& 562.32\\
3&1.87954& 1.87958 & 1.80897& 1.80899 & 1.79887& 1.79893&  1.78543 (1.77387, 1.79699) & 780.25\\
4&2.84646 & 2.84650 & 2.75748& 2.75751 & 2.74478 & 2.74486 &   2.74978 (2.73299, 2.76656) &  1002.71\\
\hline
\end{tabular}   
\end{tabular} \\
(iv) Uniform$(0,2)$ \\
\caption{Computation for $f^{(lin)}$ via scale function and simulation.} \label{results_lin}
\end{table}

\begin{table}[ht]
\begin{tabular}{c}
\centering 
\begin{tabular}{|c|rr|rr|rr|cr|} 
\hline
& \multicolumn{2}{c|}{$m=3$}& \multicolumn{2}{c|}{$m=9$} & \multicolumn{2}{c|}{$m=15$} & \multicolumn{2}{c|}{simulation}\\ 
\hline 
$x$ & \multicolumn{1}{c}{Regular} & \multicolumn{1}{c|}{Coxian} & \multicolumn{1}{c}{Regular} & \multicolumn{1}{c|}{Coxian} & \multicolumn{1}{c}{Regular} & \multicolumn{1}{c|}{Coxian} & \multicolumn{1}{c}{mean (95\% CI)} & \multicolumn{1}{c|}{time} \\
\hline
1&3.91327& 3.91300 & 3.85866& 3.85787 & 3.85841 & 3.85843 &    3.99930 (3.93021, 4.06839) & 653.32\\
2&8.86274& 8.86231  &8.77410& 8.77262&  8.77316& 8.77343 &  8.97656 (8.88100,  9.07211)& 1081.60\\
3&15.35594&  15.35535 & 15.24146& 15.23950 & 15.23986&  15.24047 &  15.40826 (15.29742, 15.51910) & 1509.49\\
4&21.51320 & 21.51259 & 21.39477& 21.39263 & 21.39282 &  21.39357 &   21.59354 (21.46230, 21.72478) & 1923.83\\
\hline
\end{tabular}   
\end{tabular} \\
(i) Normal$(0,1)$ \\
\begin{tabular}{c}
\centering 
\begin{tabular}{|c|rr|rr|rr|cr|} 
\hline
& \multicolumn{2}{c|}{$m=3$}& \multicolumn{2}{c|}{$m=9$} & \multicolumn{2}{c|}{$m=15$} & \multicolumn{2}{c|}{simulation}\\ 
\hline 
$x$ & \multicolumn{1}{c}{Regular} & \multicolumn{1}{c|}{Coxian} & \multicolumn{1}{c}{Regular} & \multicolumn{1}{c|}{Coxian} & \multicolumn{1}{c}{Regular} & \multicolumn{1}{c|}{Coxian} & \multicolumn{1}{c}{mean (95\% CI)} & \multicolumn{1}{c|}{time} \\
\hline
1& 2.44576 & 2.44567  & 2.32883 & 2.32850 & 2.32600  & 2.32767    &    2.32949 ( 2.29418, 2.36479)& 596.42\\
2& 6.04544    & 6.04529   & 5.84960   & 5.84906  & 5.84476 & 5.84764   &  5.90037 (5.85258, 5.94816)& 1017.14\\
3&11.14758& 11.14741  &  10.89334 & 10.89257   & 10.88642   & 10.89054  &   10.86792 (10.79746, 10.93838)&  1421.04\\
4&16.05701 & 16.05684 &15.80446 & 15.80367 & 15.79719 & 15.80152 &  15.84273 (15.74270
, 15.94275) & 1841.28\\
\hline
\end{tabular}   
\end{tabular} \\
(ii) Weibull$(2,1)$ \\
\begin{tabular}{c}
\centering 
\begin{tabular}{|c|rr|rr|rr|cr|} 
\hline
& \multicolumn{2}{c|}{$m=3$}& \multicolumn{2}{c|}{$m=9$} & \multicolumn{2}{c|}{$m=15$} & \multicolumn{2}{c|}{simulation}\\ 
\hline 
$x$ & \multicolumn{1}{c}{Regular} & \multicolumn{1}{c|}{Coxian} & \multicolumn{1}{c}{Regular} & \multicolumn{1}{c|}{Coxian} & \multicolumn{1}{c}{Regular} & \multicolumn{1}{c|}{Coxian} & \multicolumn{1}{c}{mean (95\% CI)} & \multicolumn{1}{c|}{time} \\
\hline
1& 1.43699  & 1.43698 & 1.36129 & 1.35954& 1.36115  &  1.36063 &    1.34876 ( 1.32836, 1.36916) & 329.38\\
2&3.68388&  3.68387   &3.55855  &  3.55371 & 3.55839 & 3.55644  & 3.55338
 (3.52135, 3.58541)& 528.42\\
3&6.99755& 6.99755 & 6.85154& 6.84365   & 6.85203& 6.84845  &  6.86599 (6.82434,  6.90764) & 732.15\\
4&10.03517 & 10.03516 &  9.90912&   9.89892  & 9.91017 & 9.90498 & 9.93187 (9.87885, 9.98490) &  931.89\\
\hline
\end{tabular}   
\end{tabular} \\
(iii) Lognormal$(0,0.5)$ \\
\begin{tabular}{c}
\centering 
\begin{tabular}{|c|cc|cc|cc|cc|} 
\hline
& \multicolumn{2}{c|}{$m=3$}& \multicolumn{2}{c|}{$m=9$} & \multicolumn{2}{c|}{$m=15$} & \multicolumn{2}{c|}{simulation}\\ 
\hline 
$x$ & Regular & Coxian & Regular & Coxian & Regular & Coxian & mean (95\% CI) & time \\
\hline
1&2.01677 & 2.01686 & 1.87440& 1.87446  & 1.85517& 1.85533 &   1.87924 (1.84889, 1.90959)& 389.59\\
2&4.93011  & 4.93025  & 4.68672& 4.68683   & 4.65362  & 4.65387 &  4.67281 (4.62459, 4.72103)&  640.53\\
3&9.03815& 9.03830 & 8.69980& 8.69987 & 8.65039& 8.65063&  8.64427 (8.58724, 8.70129) & 896.49\\
4&12.83764 & 12.83778 & 12.49212 & 12.49218 & 12.44094 & 12.44114 & 12.46704 (12.39893, 12.53514) &  1144.27\\
\hline
\end{tabular}   
\end{tabular} \\
(iv) Uniform$(0,2)$ \\
\caption{Computation for $f^{(exp)}$ via scale function and simulation.} \label{results_exp}
\end{table}

\begin{table}[ht] 
\begin{tabular}{c}
\centering 
\begin{tabular}{|c|rr|rr|rr|cr|} 
\hline
& \multicolumn{2}{c|}{$m=3$}& \multicolumn{2}{c|}{$m=9$} & \multicolumn{2}{c|}{$m=15$} & \multicolumn{2}{c|}{simulation}\\ 
\hline 
$x$ & \multicolumn{1}{c}{Regular} & \multicolumn{1}{c|}{Coxian} & \multicolumn{1}{c}{Regular} & \multicolumn{1}{c|}{Coxian} & \multicolumn{1}{c}{Regular} & \multicolumn{1}{c|}{Coxian} & \multicolumn{1}{c}{mean (95\% CI)} & \multicolumn{1}{c|}{time} \\
\hline
1&-3.13582& -3.13563 & -3.11807 & -3.11781 & -3.11731 & -3.11789  &    -3.12326 (-3.14981, -3.09670) & 601.13\\
2&-3.38191& -3.38189 & -3.39678& -3.39687  & -3.39576&  -3.39653 &  -3.42146 (-3.46281, -3.38012)&  1004.57\\
3&-0.35763& -0.35785 & -0.41582& -0.41665 & -0.41551 & -0.41595  &  -0.41487 (-0.46512,  -0.36462) &  1396.95\\
4&2.91984 &  2.91947 &2.83341& 2.83191 & 2.83266 & 2.83272 &  2.85455 (2.78751, 2.92159) &  1812.60\\
\hline
\end{tabular}   
\end{tabular} \\
(i) Normal$(0,1)$ \\
\begin{tabular}{c}
\centering 
\begin{tabular}{|c|rr|rr|rr|cr|} 
\hline
& \multicolumn{2}{c|}{$m=3$}& \multicolumn{2}{c|}{$m=9$} & \multicolumn{2}{c|}{$m=15$} & \multicolumn{2}{c|}{simulation}\\ 
\hline 
$x$ & \multicolumn{1}{c}{Regular} & \multicolumn{1}{c|}{Coxian} & \multicolumn{1}{c}{Regular} & \multicolumn{1}{c|}{Coxian} & \multicolumn{1}{c}{Regular} & \multicolumn{1}{c|}{Coxian} & \multicolumn{1}{c}{mean (95\% CI)} & \multicolumn{1}{c|}{time} \\
\hline
1& -2.88841& -2.88841  & -2.84207   & -2.84179   & -2.83974  & -2.84112   &   -2.86276 (-2.88193, -2.84360) & 551.23\\
2& -3.63571& -3.63578  & -3.68150 & -3.68151 & -3.68150 & -3.68153    &   -3.69096 ( -3.72003, -3.66189)&  946.72\\
3&-1.49099 & -1.49114 & -1.65005  &  -1.65038  & -1.65290  & -1.65122   &  
  -1.66218 (-1.69899
, -1.62537) & 1363.74\\
4&1.09153 & 1.09137 & 0.89409 &  0.89351 & 0.88877  & 0.89192 & 0.84962
 (0.80292, 0.89633) &  1736.75\\
\hline
\end{tabular}    
\end{tabular} \\
(ii) Weibull$(2,1)$ \\
\begin{tabular}{c}
\centering 
\begin{tabular}{|c|rr|rr|rr|cr|} 
\hline
& \multicolumn{2}{c|}{$m=3$}& \multicolumn{2}{c|}{$m=9$} & \multicolumn{2}{c|}{$m=15$} & \multicolumn{2}{c|}{simulation}\\ 
\hline 
$x$ & \multicolumn{1}{c}{Regular} & \multicolumn{1}{c|}{Coxian} & \multicolumn{1}{c}{Regular} & \multicolumn{1}{c|}{Coxian} & \multicolumn{1}{c}{Regular} & \multicolumn{1}{c|}{Coxian} & \multicolumn{1}{c}{mean (95\% CI)} & \multicolumn{1}{c|}{time} \\
\hline
1&-2.17768& -2.17768 & -2.16212 & -2.16047    & -2.16294    & -2.16239  &    -2.16779 (-2.18408,  -2.15151) & 310.62\\
2&-2.78953& -2.78953  &  -2.84883&  -2.84639 & -2.85014   & -2.84832  &   -2.87129 (-2.88902, -2.85356)& 498.01\\
3&-1.21537& -1.21537 & -1.35601 &-1.35655 & -1.35748& -1.35669 &  -1.38980
 (-1.41319,  -1.36642) & 687.00\\
4&0.36719 & 0.36718 &  0.25258&  0.24716 &  0.25190 & 0.25038 & 0.21037 (0.18165, 0.23908) & 870.40\\
\hline
\end{tabular}   
\end{tabular} \\
(iii) Lognormal$(0,0.5)$ \\
\begin{tabular}{c}
\centering 
\begin{tabular}{|c|rr|rr|rr|cr|} 
\hline
& \multicolumn{2}{c|}{$m=3$}& \multicolumn{2}{c|}{$m=9$} & \multicolumn{2}{c|}{$m=15$} & \multicolumn{2}{c|}{simulation}\\ 
\hline 
$x$ & \multicolumn{1}{c}{Regular} & \multicolumn{1}{c|}{Coxian} & \multicolumn{1}{c}{Regular} & \multicolumn{1}{c|}{Coxian} & \multicolumn{1}{c}{Regular} & \multicolumn{1}{c|}{Coxian} & \multicolumn{1}{c}{mean (95\% CI)} & \multicolumn{1}{c|}{time} \\
\hline
1&-2.51479 & -2.51475 & -2.41060& -2.41049  & -2.39284 & -2.39284 &   -2.41128 (-2.42767, -2.39488)& 365.04\\
2&-3.05340  & -3.05330  & -3.06535& -3.06524   & -3.06739 & -3.06727 &  -3.11731 (-3.13792, -3.09670)& 606.77\\
3&-1.08817& -1.08801 & -1.24750& -1.24732 & -1.26687 &  -1.26658&  -1.29956 (-1.32655, -1.27258) & 849.17\\
4&0.90986 & 0.91001 & 0.65974& 0.65981 & 0.61977 & 0.61992 &   0.58704 (0.55266,  0.62143) & 1098.11\\
\hline
\end{tabular}   
\end{tabular} \\
(iv) Uniform$(0,2)$ \\
\caption{Computation for $f^{(sim)}$ via scale function and simulation.} \label{results_sim}
\end{table}

\section{Concluding Remarks}
We have studied the scale function for the spectrally negative PH \lev process and the PH-fitting approach for the approximation of the scale function for a general spectrally negative \lev process.  Because the fitted scale function is given as a function in a closed form, one can analytically obtain other fluctuation identities explicitly. Our numerical results based on the EM-algorithm of \cite{Asmussen_1996} suggest that the PH-fitting is a powerful approximation tool  at least when one is interested in obtaining it in an analytical form.

While our numerical results already exhibit reasonable accuracy of the PH-fitting of scale functions, there is still a room for improvement.  There exist a variety of fitting algorithms typically developed in queueing analysis.
Well-known examples are the moment-matching approach (e.g.\ MEFIT and MEDA) and the maximum-likelihood approach (e.g.\ MLAPH and EMPHT), and a thorough study of pros and cons of each fitting techniques has been conducted in, for example, \cite{Horvath_Telek, Lang_Arthur_1996}.  Our next step is, therefore, to apply these existing algorithms for the approximation of the scale function, and analyze their performance for a variety of \lev measures.

Finally, the PH-fitting construction of scale functions can also be achieved from empirical data as in, among others, \cite{AsmussenMadanPistorius07}. The closed-form expression of the approximated scale function can be used flexibly to identify the fluctuation of the process implied by the empirical data.


\appendix

\section{Fitted Data}

The fitted PH-distributions and parameters of the corresponding scale functions for $m = 3,9$ are given below.

\subsection{Fitted PH-distributions}  Let  $(m, \bm \alpha_m, \bm T_m)$ and $(m, [1,0,\ldots, 0], \bm T_m^c)$ be the fitted PH-distributions for the regular and Coxian fit, respectively.

(i) \underline{Normal$(0,1)$}:
\begin{align*}
&{\bm T}_3 = \left[ \begin{array}{rrr}   
   -2.8330  &  0.0000   & 0.0000 \\
    2.3613  & -2.7743 &   0.0000 \\
    0.1073  &  2.1292 &  -2.8454 
\end{array} \right], 
&{\bm T}_3^c = \left[ \begin{array}{rrr}   
   -2.8512    & 2.0459  &       0 \\
         0   &-2.7676 &   2.0926 \\
         0       &  0   &-2.8400 
\end{array} \right], \\ &{\bm \alpha}_3 = [ \begin{array}{rrr}  0.0924   &  0.0578 &   0.8497\end{array} ],   
\end{align*}
\begin{align*} 
&{\bm T}_9 = \left[ \begin{array}{rrrrrrrrr}   
   -3.7115 &   0.0000  &  0.0000   & 0.0000    &0.0002 &   0.0241    &3.4579   & 0.0000  &  0.0000 \\
    0.6513  & -4.5953  &  0.0974   & 0.0002    &0.6422   & 0.4061    &0.0319    &0.5389 &   2.0614 \\
    0.0534   & 0.5135  & -5.9112    &0.0310    &0.7069    &0.1432    &0.3790    &0.4071  &  0.2037 \\
    0.0050   & 0.7520  &  1.6122   &-4.7124   & 0.1050    &0.1331    &0.2998    &1.5375 &   0.1023 \\
    0.7898  &  0.0503  &  0.0549    &0.0000   &-4.8933    &0.7107    &1.0245    &0.0246 &   1.3825 \\
    2.3764   & 0.0001  &  0.0017    &0.0000    &0.0477   &-4.3859   & 0.9693    &0.0000 &   0.0258 \\
    0.0001   & 0.0000   & 0.0000    &0.0000   & 0.0000    &0.0000   &-3.7831 &   0.0000  &  0.0000 \\
    0.6489  &  0.6410  &  0.0196    &0.0002  &  0.3457    &0.3718  &  0.0226  & -4.1770 &   2.0345 \\
    3.2096  &  0.0006 &   0.0017    &0.0000&    0.0960    &0.6220 &   0.0870  &  0.0004 &  -4.0823
 \end{array} \right],  \\ &{\bm \alpha}_9 =    [ \begin{array}{rrrrrrrrr}   0.0251   & 0.0209  &  0.1089 &   0.7375&    0.0521 &   0.0219 &   0.0164  &  0.0147 &   0.0026 \end{array} ],
\end{align*}
\begin{align*} 
&{\bm T}^c_9 = \left[ \begin{array}{rrrrrrrrr}   
-4.0461   & 3.4027  &       0    &     0    &     0     &    0     &    0   &     0     &    0 \\
         0  & -4.3285   & 2.7480    &     0  &       0   &     0       &  0    &     0 &        0 \\
         0    &     0  & -3.5110  &  3.3597    &     0  &       0    &     0   &      0     &    0 \\
         0     &    0    &     0  & -3.5482  &  1.5160    &     0    &     0     &    0     &    0 \\
         0     &    0   &      0     &    0 &  -4.5039  &  1.9858  &       0    &     0    &     0 \\
         0    &     0    &     0   &      0    &     0  & -6.7093   & 0.5874   &      0    &     0 \\
         0      &   0     &    0    &     0    &     0    &     0  &-11.9250  &  0.2777     &    0 \\
         0       &  0      &   0    &     0    &     0    &     0   &      0  &-18.1680  &  1.0444 \\
         0      &   0   &      0    &     0   &      0    &     0    &     0     &    0  & -23.5959
 \end{array} \right].
\end{align*}
(ii) \underline{Weibull$(2,1)$}:
\begin{align*}
&{\bm T}_3 = \left[ \begin{array}{rrr}   
   -2.8330  &  0.0000   & 0.0000 \\
    2.3613  & -2.7743 &   0.0000 \\
    0.1073  &  2.1292 &  -2.8454 
\end{array} \right], 
&{\bm T}_3^c = \left[ \begin{array}{rrr}   
   -3.3439  &  3.3439   &      0 \\
         0 &  -3.3439 &   3.2273 \\
         0  &       0   &-3.3499 \\
\end{array} \right], \\ &{\bm \alpha}_3 = [ \begin{array}{rrr}  0.0000  &  0.0125 &   0.9875 \end{array} ],
\end{align*}
\begin{align*} 
&{\bm T}_9 = \left[ \begin{array}{rrrrrrrrr}   
   -6.3782    &0.0000  &  0.0000  &  0.0006  &  0.0000 &   0.0000  &  0.7112  &  5.6547  &  0.0000\\
    0.2809   &-6.3362 &   2.3813  &  3.2971 &   0.0082 &   0.0000  &  0.2241 &  0.1440  &  0.0000\\
    2.7488  &  0.0015  & -6.0787 &   3.1742 &   0.0000 &   0.0000 &   0.0484 &   0.0942  &  0.0000\\
    2.5080 &   0.0000 &   0.0001 &  -6.1893  &  0.0000 &   0.0000 &   0.0418  &  3.6395 &   0.0000\\
    0.0679 &   0.5571  &  3.8628 &   0.1807 &  -6.1770 &   0.0000&    1.4364 &   0.0485&    0.0236\\
    0.0388  &  0.2401 &   0.0104 &   0.0038&    1.5906  & -5.8760  &  0.1887 &   0.0356 &   3.7679\\
    0.0000  &  0.0000   & 0.0000&    0.0000 &   0.0000 &   0.0000  & -5.8447 &   0.0000 &   0.0000\\
    0.0001  &  0.0000  &  0.0000 &   0.0000  &  0.0000 &   0.0000 &  5.8605  & -5.8832   & 0.0000\\
    0.0275 &   2.5665 &   1.4770  &  0.0165  &  0.8888  &  0.0000  &  1.2733  &  0.0135  & -6.2632
 \end{array} \right],  \\ &{\bm \alpha}_9 =    [ \begin{array}{rrrrrrrrr}   0.0015  &  0.0031 &   0.0009 &   0.0134    &0.0052 &   0.9485   & 0.0000  &  0.0000 &   0.0274 \end{array} ], \\
&{\bm T}^c_9 = \left[ \begin{array}{rrrrrrrrr}   
   -5.8363 &   5.8363    &     0     &    0  &       0    &     0     &    0 &        0    &     0 \\
         0   &-5.8363    &5.6125   &      0  &       0    &     0   &      0    &    0       &  0 \\
         0    &     0  & -5.8863 &  4.5513      &   0     &    0     &    0   &      0   &     0 \\
         0    &     0   &      0  & -5.8776  &  5.6237    &     0     &    0     &    0   &      0 \\
         0      &   0      &   0  &       0  & -5.8454  &  4.7530    &     0  &       0     &    0 \\
         0     &    0   &      0      &   0    &     0 &  -6.0098  &  2.7380   &      0  &       0 \\
         0    &     0   &      0  &       0    &     0    &     0 &  -7.4417  &  0.1083   &      0 \\
         0   &      0   &      0     &    0    &     0   &      0   &      0 & -13.4635  &  0.1481 \\
         0     &    0 &        0    &     0 &        0   &      0  &       0  &       0 & -24.4108
 \end{array} \right].
\end{align*}
(iii) \underline{Lognormal$(0,0.5)$}:
\begin{align*}
&{\bm T}_3 = \left[ \begin{array}{rrr}   
   -2.6473 &   2.6473  &  0.0000 \\
    0.0000  & -2.6564  &  0.0000 \\
    2.6389  &  0.0000  & -2.6389
\end{array} \right], 
&{\bm T}_3^c = \left[ \begin{array}{rrr}   
   -2.6410   &  2.6410 &        0 \\
         0   &-2.6420    &2.6420 \\
         0     &    0   &-2.6595
\end{array} \right], \\ &{\bm \alpha}_3 = [\begin{array}{rrr}  0.0000 &   0.0000  &  1.0000\end{array} ],
\end{align*}
\begin{align*} 
&{\bm T}_9 = \left[ \begin{array}{rrrrrrrrr}   
    -8.8578  &  0.0286   & 0.0029   & 0.0888  &  0.0047  &  7.1335   & 0.0170  &  1.5686   & 0.0138\\
    0.0000 &  -7.5274  &  5.4631 &   0.0000  &  0.0000 &   0.0000 &   1.5104 &   0.0000  &  0.5539\\
    0.0001  &  0.0007 &  -6.6246 &   6.0864  &  0.0000  &  0.0000&    0.5362  &  0.0000  &  0.0012\\
    1.3151  &  0.0003 &   0.0009 &  -7.0777 &   0.0000  &  0.0210  &  0.0093  &  5.7310  &  0.0002\\
    0.0004 &  0.4254 &   0.0011  &  0.0004 &  -5.4666   & 0.0146 &   0.0012  &  0.0006 &   0.1851\\
    0.0004  &  0.0498  &  0.0002 &   0.0003 &   6.9732   &-7.0421 &   0.0011  &  0.0042 &   0.0130\\
    0.0067  &  0.0047 &   0.0750  &  2.8372  &  0.0000  &  0.0000  & -2.9367  &  0.0107 &   0.0025\\
    0.0767  &  0.0009  &  0.0001  &  0.0100  &  0.0001   &10.1447 &   0.0005 & -10.2336 &  0.0008\\
    0.0000 &   0.7020 &   3.1540  &  0.0000   & 0.0000   & 0.0000  &  0.7259   & 0.0000 &  -4.5819
 \end{array} \right],  \\ &{\bm \alpha}_9 =    [\begin{array}{rrrrrrrrr}   0.0000 &   0.2681  &  0.0000 &  0.0000     &    0      &   0   &0.0000 &   0.0000  &  0.7319 \end{array} ], \\
&{\bm T}^c_9 = \left[ \begin{array}{rrrrrrrrr}   
    -1.8290   & 1.8290   &      0   &      0   &      0  &      0    &     0   &      0     &    0 \\
         0   &-9.5931  &  9.5931    &     0    &     0    &     0    &     0   &      0   &      0 \\
         0     &    0  &-9.5931  &  9.5931    &     0    &     0    &     0    &     0     &    0 \\
         0   &      0    &     0   &-9.5931  &  9.5931   &      0     &    0   &      0    &     0 \\
         0   &      0    &     0    &     0   &-9.5932   & 9.5932    &     0   &      0   &      0 \\
         0   &      0    &     0    &     0    &     0  & -9.6481  &  9.2726   &      0   &      0 \\
         0   &      0     &    0    &     0    &     0   &      0  &-14.8141 &   0.5492   &      0 \\
         0   &      0    &     0   &      0    &     0  &       0    &     0  &-38.8187  &  0.3365 \\
         0   &      0    &     0   &      0    &     0   &      0   &      0   &      0  & -67.0421
 \end{array} \right].
\end{align*}
(iv) \underline{Uniform$(0,2)$}:
\begin{align*}
&{\bm T}_3 = \left[ \begin{array}{rrr}   
   -2.5729  &  2.5448&    0.0000\\
    0.0000  & -2.5645&    2.5644 \\
    0.0000   & 0.0000 &  -2.7844
\end{array} \right], 
&{\bm T}_3^c = \left[ \begin{array}{rrr}   
   -2.7978  &  2.3130   &      0 \\
         0   &-2.5639   & 2.5637 \\
         0     &    0  & -2.5640 
\end{array} \right], \\ &{\bm \alpha}_3 = [ \begin{array}{rrr}  0.8348 &   0.0000  &  0.1652\end{array} ],   
\end{align*}
\begin{align*} 
&{\bm T}_9 = \left[ \begin{array}{rrrrrrrrr}   
   -6.3869   & 0.0000 &   0.0000 &   6.3869  &  0.0000 &   0.0000   & 0.0000   & 0.0000 &   0.0000 \\
    6.2800  & -6.4310  &  0.0000 &   0.0000  &  0.0000  &  0.0000   & 0.0000   & 0.1510   & 0.0000 \\
    0.0000  &  5.5196  & -6.4878 &   0.0000 &   0.0022  &  0.0000   & 0.0000 &   0.1180  &  0.0000 \\
    0.0000  &  0.0000  &  0.0000 &  -6.3869  &  0.0000  &  0.0000  &  6.3869  &  0.0000  &  0.0000 \\
    0.0000   &      0    &     0  &  0.0000  & -6.9388  &  0.0000  &  0.0000   & 6.9388     &    0 \\
    0.0000   &      0     &    0  &  0.0000  &  6.9381 &  -6.9381  &  0.0000  &  0.0000     &    0 \\
    0.0000   & 0.0000  &  0.0000  &  0.0000 &   0.0000   & 6.3929  & -6.3929  &  0.0000    &     0 \\
         0     &    0     &    0    0.0000  &  0.0000  &  0.0000  &  0.0000  & -6.9766  &       0 \\
    0.0000 &  0.0000   & 5.4945 &   0.0000 &   0.0008 &   1.3999  &  0.0071   & 0.0000 &  -7.0566
 \end{array} \right],  \\ &{\bm \alpha}_9 =    [ \begin{array}{rrrrrrrrr}   0.0000  &  0.0000&   0.0000   & 0.0000    0.0000   & 0.0000  &  0.0004   & 0.0400 &   0.9596 \end{array} ], \\
&{\bm T}^c_9 = \left[ \begin{array}{rrrrrrrrr}   
   -7.0218   & 6.5900   &      0  &      0     &    0     &    0    &     0 &        0   &     0 \\
         0 & -7.0341  &  6.3763   &      0   &      0     &    0    &     0   &      0      &   0 \\
         0    &     0  & -6.9384  &  6.7546     &    0    &     0    &     0    &     0    &     0 \\
         0    &     0   &      0  & -6.9651  &  5.2304    &     0    &     0   &      0   &      0 \\
         0     &    0  &      0    &     0 &  -6.4125   & 6.4124   &      0    &     0     &    0 \\
         0     &    0   &      0   &      0   &      0   &-6.4125  &  6.4125   &      0   &      0 \\
         0     &    0   &     0    &     0   &      0   &      0  & -6.4125  &  6.4125   &      0 \\
         0    &     0  &       0   &      0  &       0   &      0    &     0  & -6.4125&    6.4125 \\
         0    &     0   &      0    &     0  &       0   &      0    &     0    &     0 &  -6.4125
 \end{array} \right].
\end{align*}

\subsection{The values of elements in $\mathcal{J}_q$}
%
\begin{center}
\begin{tabular}{|c|rr|rr|rr|rr|}
\hline
 & \multicolumn{2}{c|}{(i) Normal} & \multicolumn{2}{c|}{(ii) Weibull} & \multicolumn{2}{c|}{(iii) Lognormal} & \multicolumn{2}{c|}{(iv) Uniform} \\
\hline 
  &    \multicolumn{1}{c}{Regular} &  \multicolumn{1}{c|}{Coxian}&  \multicolumn{1}{c}{Regular} &  \multicolumn{1}{c|}{Coxian} &   \multicolumn{1}{c}{Regular} &  \multicolumn{1}{c|}{Coxian} &  \multicolumn{1}{c}{Regular} &  \multicolumn{1}{c|}{Coxian} \\
\hline
$\eta_{1,q}$   &   2.7743 & 2.7676 & 3.3432 & 3.3439 & 2.6389   & 2.6410&  2.5645 &2.5639  \\
$\eta_{2,q}$   &  2.8331 & 2.8400 & 3.3445 & 3.3439 &  2.6395 &2.6420 & 2.5729  &2.5640\\
$\eta_{3,q}$  & 2.8454 & 2.8512 &  3.3492& 3.3439  & 2.6642  &2.6595 & 2.7844& 2.7978\\
  \hline
\end{tabular} \\ 
$m =3$ \\ \vspace{0.2cm}
\begin{tabular}{|c|cc|cccc|cccc|cc|}
\hline
 & \multicolumn{2}{c|}{(i) Normal} & \multicolumn{4}{c|}{(ii) Weibull} & \multicolumn{4}{c|}{(iii) Lognormal} & \multicolumn{2}{c|}{(iv) Uniform} \\
\hline 
 &    \multicolumn{1}{c}{Regular} &  \multicolumn{1}{c|}{Coxian}&  \multicolumn{3}{c}{Regular} &  \multicolumn{1}{c|}{Coxian} &   \multicolumn{3}{c}{Regular} &  \multicolumn{1}{c|}{Coxian} &  \multicolumn{1}{c}{Regular} &  \multicolumn{1}{c|}{Coxian} \\
\hline
$\eta_{1,q}$   &   3.5296&  3.5110 &  5.8758 &+& 0.0000i &  5.8363 &   1.7725 &+& 0.0000i & 1.8290&   6.3869 & 6.4125\\
$\eta_{2,q}$   & 3.7845   & 3.5482 &  5.8447 &+& 0.0000i  &   5.8363 &  3.6157 &+ &0.0000i   &9.5931&6.3869  &6.4125\\
$\eta_{3,q}$  &  3.6990 & 4.0461  &  5.8821 &+& 0.0000i &   5.8454  &  4.5796 &+& 3.6631i  & 9.5931 &6.3929 &  6.4125\\
$\eta_{4,q}$   &  4.0039 & 4.3285& 6.0398 &+ &0.0000i &  5.8776&   4.5796 &-& 3.6631i &  9.5931&  6.4310 &  6.4125 \\
$\eta_{5,q}$   & 4.5280    & 4.5039 &  6.0774& + &0.0000i &  5.8863&  5.6786&+& 0.0000i& 9.5932 & 6.4878 & 6.4125\\
$\eta_{6,q}$  & 4.6833&6.7093  & 6.1836 &+ &0.0000i & 6.0098  & 9.3981 &+& 0.0000i  &9.6481 &  6.9379&  6.9384 \\
$\eta_{7,q}$   &  4.9074 & 11.9250 & 6.3679 &+ &0.0068i  & 7.4417 &  9.4126 &+ &3.5997i   & 14.8141& 6.9389  &  6.9651 \\
$\eta_{8,q}$   &  5.1110 & 18.1680 & 6.3679 &-& 0.0068i & 13.4635 & 9.4126& -& 3.5997i   &38.8187 &  6.9766& 7.0218\\
$\eta_{9,q}$  &  6.0054 &23.5959  & 6.3874 &+& 0.0000i  &  24.4108  & 11.8990& +& 0.0000i &67.0421 &7.0566& 7.0341 \\
  \hline
\end{tabular} 
$m = 9$
\end{center}
\subsection{The roots of $\mathcal{I}_q$ and $\Phi(q)$}

\underline{(a) $\lambda = 5$ and $\sigma = 1$:} \begin{center}
\begin{tabular}{|c|rcr|rcr|rcr|rcr|}
\hline
 & \multicolumn{6}{c|}{(i) Normal} & \multicolumn{6}{c|}{(ii) Weibull}\\
\hline 
 &    \multicolumn{3}{c|}{Regular} &  \multicolumn{3}{c|}{Coxian}&  \multicolumn{3}{c|}{Regular} &  \multicolumn{3}{c|}{Coxian}  \\
\hline
$\Phi(q)$   &  \multicolumn{3}{c|}{0.0443} &\multicolumn{3}{c|}{0.0443}&\multicolumn{3}{c|}{ 0.0651}& \multicolumn{3}{c|}{0.0651}\\
\hline
$\xi_{1,q}$   & 0.3197 &+& 0.0000i   &0.3197 &+& 0.0000i& 0.2276 &+& 0.0000i& 0.2276 &+& 0.0000i\\
$\xi_{2,q}$   & 3.5847 &+& 1.2168i &3.5874 &+& 1.2202i & 4.4438 &+& 1.8003i &  4.4441 &+& 1.8006i \\
$\xi_{3,q}$  &3.5847 &-& 1.2168i &3.5874& -& 1.2202i&  4.4438 &-& 1.8003i&  4.4441 &-& 1.8006i\\
$\xi_{4,q}$   &11.0080& + &0.0000i& 11.0086 &+& 0.0000i&10.9869 &+& 0.0000i&  10.9870 &+& 0.0000i\\
  \hline
\end{tabular} \\ 
$m =3$ \\ \vspace{0.2cm}
\begin{tabular}{|c|rcr|rcr|rcr|rcr|}
\hline
 & \multicolumn{6}{c|}{(i) Normal} & \multicolumn{6}{c|}{(ii) Weibull}\\
\hline 
 &    \multicolumn{3}{c|}{Regular} &  \multicolumn{3}{c|}{Coxian}&  \multicolumn{3}{c|}{Regular} &  \multicolumn{3}{c|}{Coxian}  \\
\hline
$\Phi(q)$   &  \multicolumn{3}{c|}{0.0443} &\multicolumn{3}{c|}{0.0443}&\multicolumn{3}{c|}{0.0656}& \multicolumn{3}{c|}{0.0656}\\
\hline
$\xi_{1,q}$   &  0.3248 &+& 0.0000i &0.3249 &+& 0.0000i& 0.2365 &+& 0.0000i& 0.2365 &+& 0.0000i\\
$\xi_{2,q}$   & 3.3380 &+& 2.2064i & 3.3099 &+& 2.2437i&3.7493 &+& 3.5397i& 3.7374 &+& 3.5516i\\
$\xi_{3,q}$  & 3.3380 &-& 2.2064i  &3.3099 &-& 2.2437i& 3.7493 &-& 3.5397i&3.7374 &-& 3.5516i\\
$\xi_{4,q}$   & 4.7176 &+& 0.0000i &5.7446 &+& 1.6149i&6.3675 &+& 0.0000i&  7.1628 &+& 3.9282i\\
$\xi_{5,q}$  &   5.0105 &+& 0.0354i  &5.7446 &-& 1.6149i & 7.2326 &+& 3.7784i&  7.1628 &-& 3.9282i\\
$\xi_{6,q}$  & 5.0105 &-& 0.0354i  & 7.3804 &+& 0.0000i&7.2326 &-& 3.7784i&  9.6617 &+& 2.1047i\\
$\xi_{7,q}$    &  5.5461 &+& 1.2267i & 10.8739 &+& 0.0000i& 7.2882 &+& 0.0000i& 9.6617 &-& 2.1047i\\
$\xi_{8,q}$   & 5.5461 &-& 1.2267i & 11.9282 &+& 0.0000l&9.1630 &+& 1.1757i&  11.4428 &+& 0.0000i\\
$\xi_{9,q}$    &  6.5805 &+& 0.0000i  &18.1680 &+& 0.0000i &9.1630 &-& 1.1757i& 13.4595 &+& 0.0000i\\
$\xi_{10,q}$  & 10.8843 &+& 0.0000i &23.5959 &+& 0.0000i &10.9100& + &0.0000i&  24.4108 &+& 0.0000i\\
  \hline
\end{tabular} \\ 
$m = 9$ \vspace{0.2cm}
\end{center}
\begin{center}
\begin{tabular}{|c|rcr|rcr|rcr|rcr|}
\hline
 & \multicolumn{6}{c|}{(iii) Lognormal} & \multicolumn{6}{c|}{(iv) Uniform}\\
\hline 
 &    \multicolumn{3}{c|}{Regular} &  \multicolumn{3}{c|}{Coxian}&  \multicolumn{3}{c|}{Regular} &  \multicolumn{3}{c|}{Coxian}  \\
\hline
$\Phi(q)$   &  \multicolumn{3}{c|}{0.2112} &\multicolumn{3}{c|}{0.2112}&\multicolumn{3}{c|}{0.1142}& \multicolumn{3}{c|}{0.1142}\\
\hline
$\xi_{1,q}$  &0.0537 &+& 0.0000i& 0.0537 &+& 0.0000i& 0.1049 &+& 0.0000i&0.1049 &+& 0.0000i\\
$\xi_{2,q}$  & 3.5742 &+& 1.4680i & 3.5741 &+& 1.4680i& 3.4677 &+& 1.3838i & 3.4695 &+& 1.3849i\\
$\xi_{3,q}$ & 3.5742 &-& 1.4680i&3.5741 &-& 1.4680i& 3.4677 &-& 1.3838i&3.4695 &-& 1.3849i\\
$\xi_{4,q}$  & 10.9518 &+& 0.0000i& 10.9518 &+& 0.0000i& 10.9957 &+& 0.0000i& 10.9960 &+& 0.0000i\\
  \hline
\end{tabular} \\ 
$m =3$ \\ \vspace{0.2cm}
\begin{tabular}{|c|rcr|rcr|rcr|rcr|}
\hline
 & \multicolumn{6}{c|}{(iii) Lognormal} & \multicolumn{6}{c|}{(iv) Uniform}\\
\hline 
 &    \multicolumn{3}{c|}{Regular} &  \multicolumn{3}{c|}{Coxian}&  \multicolumn{3}{c|}{Regular} &  \multicolumn{3}{c|}{Coxian}  \\
\hline
$\Phi(q)$   &  \multicolumn{3}{c|}{0.2171} &\multicolumn{3}{c|}{0.2174}&\multicolumn{3}{c|}{0.1175}& \multicolumn{3}{c|}{0.1175}\\
\hline
$\xi_{1,q}$   & 0.0542 &+& 0.0000i&0.0542 &+& 0.0000i&0.1087 &+& 0.0000i&0.1087 &+& 0.0000i\\
$\xi_{2,q}$  &  3.5710 &+& 0.0000i &5.0976 &+& 3.9959i& 2.6930 &+& 3.5210i&2.6923 &+& 3.5228i\\
$\xi_{3,q}$  & 4.1447 &+& 4.2681i &5.0976 &-& 3.9959i & 2.6930 &-& 3.5210i&2.6923 &-& 3.5228i\\
$\xi_{4,q}$  & 4.1447 &-& 4.2681i & 10.0067 &+& 5.3918i & 5.7295 &+& 5.0992i&5.7335 &+& 5.1012i\\
$\xi_{5,q}$   &   6.7264 &+& 0.0000i&10.0067 &-& 5.3918i&5.7295 &-& 5.0992i&5.7335 &-& 5.1012i \\
$\xi_{6,q}$  &  9.1911 &+& 4.2497i & 14.2179 &+& 3.4774i  &8.8548 &+& 4.6773i&8.8564 &+& 4.6816i\\
$\xi_{7,q}$  & 9.1911 &-& 4.2497i&14.2179 &-& 3.4774i& 8.8548 &-& 4.6773i&8.8564 &-& 4.6816i\\
$\xi_{8,q}$  & 9.4031 &+& 0.0000i&16.1826 &+& 0.0000i& 11.5388 &+& 2.7343i&11.5452 &+& 2.7338i\\
$\xi_{9,q}$  & 12.0697 &+& 1.3445i&38.8187 &+& 0.0000i& 11.5388 &-& 2.7343i  &11.5452 &-& 2.7338i\\
$\xi_{10,q}$  &  12.0697 &-& 1.3445i&67.0421 &+& 0.0000i&12.3722 &+& 0.0000i& 12.3758 &+& 0.0000i \\
  \hline
\end{tabular} \\ 
$m = 9$
\end{center}
\underline{(b) $\lambda = 5$ and $\sigma = 0$:} 
\begin{center}
\begin{tabular}{|c|rcr|rcr|rcr|rcr|}
\hline
 & \multicolumn{6}{c|}{(i) Normal} & \multicolumn{6}{c|}{(ii) Weibull}\\
\hline 
 &    \multicolumn{3}{c|}{Regular} &  \multicolumn{3}{c|}{Coxian}&  \multicolumn{3}{c|}{Regular} &  \multicolumn{3}{c|}{Coxian}  \\
\hline
$\Phi(q)$   &  \multicolumn{3}{c|}{0.0451} &\multicolumn{3}{c|}{0.0451}&\multicolumn{3}{c|}{0.0674}& \multicolumn{3}{c|}{0.0674 }\\
\hline
$\xi_{1,q}$   &    0.3570 &+& 0.0000i & 0.3571 &+& 0.0000i &  0.2544 &+& 0.0000i & 0.2544 &+& 0.0000i  \\
$\xi_{2,q}$   &   3.5654 &+& 1.0773i & 3.5685 &+& 1.0801i & 4.4200 &+& 1.5069i & 4.4203 &+& 1.5072i \\
$\xi_{3,q}$  & 3.5654 &-& 1.0773i & 3.5685 &-& 1.0801i  & 4.4200 &-& 1.5069i & 4.4203 &-& 1.5072i  \\
  \hline
\end{tabular} \\ 
$m =3$ \\ \vspace{0.2cm}
\begin{tabular}{|c|rcr|rcr|rcr|rcr|}
\hline
 & \multicolumn{6}{c|}{(i) Normal} & \multicolumn{6}{c|}{(ii) Weibull}\\
\hline 
 &    \multicolumn{3}{c|}{Regular} &  \multicolumn{3}{c|}{Coxian}&  \multicolumn{3}{c|}{Regular} &  \multicolumn{3}{c|}{Coxian}  \\
\hline
$\Phi(q)$   &  \multicolumn{3}{c|}{0.0452} &\multicolumn{3}{c|}{0.0452}&\multicolumn{3}{c|}{0.0680}& \multicolumn{3}{c|}{0.0680}\\
\hline
$\xi_{1,q}$   &    0.3638& +& 0.0000i & 0.3639 &+& 0.0000i &   0.2663 &+ &0.0000i & 0.2664 &+& 0.0000i  \\
$\xi_{2,q}$   & 3.4702 &+& 2.1218i   & 3.4460 &+& 2.1692i& 4.0425 &+& 3.5128i &  4.0302 &+& 3.5303i \\
$\xi_{3,q}$  & 3.4702 &-& 2.1218i &  3.4460 &-& 2.1692i  &4.0425 &-& 3.5128i  &  4.0302 &-& 3.5303i \\
$\xi_{4,q}$   & 4.7176 &+& 0.0000i    & 5.6793 &+& 1.3918i&  6.3675 &+& 0.0000i & 7.4092 &+& 3.4299i \\
$\xi_{5,q}$   &  5.0108 &+& 0.0336i &  5.6793 &-& 1.3918i & 7.2875 &+& 0.0000i &  7.4092 &-& 3.4299i\\
$\xi_{6,q}$  &  5.0108 &-& 0.0336i  & 7.0681 &+& 0.0000i   &  7.4140 &+& 3.2249i  &   9.3226 &+& 1.3680i \\
$\xi_{7,q}$   &  5.4445 &+& 1.0984i  &11.9246 &+& 0.0000i &   7.4140 &-& 3.2249i &  9.3226 &-& 1.3680i \\
$\xi_{8,q}$   & 5.4445 &-& 1.0984i& 18.1680 &+& 0.0000i  & 8.6251 &+& 0.9483i&  13.4645 &+& 0.0000i \\
$\xi_{9,q}$  & 6.3550 &+& 0.0000i&  23.5959 &+& 0.0000i  &  8.6251 &-& 0.9483i &  24.4107 &+& 0.0000i  \\
  \hline
\end{tabular} \\ 
$m = 9$ \\ \vspace{0.2cm}
\begin{tabular}{|c|rcr|rcr|rcr|rcr|}
\hline
 & \multicolumn{6}{c|}{(iii) Lognormal} & \multicolumn{6}{c|}{(iv) Uniform}\\
\hline 
 &    \multicolumn{3}{c|}{Regular} &  \multicolumn{3}{c|}{Coxian}&  \multicolumn{3}{c|}{Regular} &  \multicolumn{3}{c|}{Coxian}  \\
\hline
$\Phi(q)$   &  \multicolumn{3}{c|}{0.2359} &\multicolumn{3}{c|}{0.2359}&\multicolumn{3}{c|}{0.1227}& \multicolumn{3}{c|}{0.1226}\\
\hline
$\xi_{1,q}$  &0.0550 &+& 0.0000i& 0.0550 &+& 0.0000i& 0.1111 &+& 0.0000i&  0.1111 &+& 0.0000i\\
$\xi_{2,q}$   &3.5568 &+& 1.2867i & 3.5568 &+& 1.2867i&3.4617 &+& 1.2230i&  3.4636 &+& 1.2238i\\
$\xi_{3,q}$ & 3.5568 &-& 1.2867i & 3.5568 &-& 1.2867i& 3.4617 &-& 1.2230i&  3.4636 &-& 1.2238i\\
  \hline
\end{tabular} \\ 
$m =3$ \\
\begin{tabular}{|c|rcr|rcr|rcr|rcr|}
\hline
 & \multicolumn{6}{c|}{(iii) Lognormal} & \multicolumn{6}{c|}{(iv) Uniform}\\
\hline 
 &    \multicolumn{3}{c|}{Regular} &  \multicolumn{3}{c|}{Coxian}&  \multicolumn{3}{c|}{Regular} &  \multicolumn{3}{c|}{Coxian}  \\
\hline
$\Phi(q)$   &  \multicolumn{3}{c|}{0.2436} &\multicolumn{3}{c|}{0.2440}&\multicolumn{3}{c|}{0.1268}& \multicolumn{3}{c|}{0.1268}\\
\hline
$\xi_{1,q}$   & 0.0554 &+& 0.0000i& 0.0555 &+& 0.0000i&0.1157 &+& 0.0000i&   0.1157 &+& 0.0000i\\
$\xi_{2,q}$   &3.5749 &+& 0.0000i&5.8096 &+& 3.9763i&2.9052 &-& 3.6310i&  2.9047 &+& 3.6333i\\
$\xi_{3,q}$  &4.4309 &+& 4.2571i&5.8096 &-& 3.9763i &2.9052 &+& 3.6310i& 2.9047 &-& 3.6333i\\
$\xi_{4,q}$  & 4.4309 &-& 4.2571i&11.0587 &+& 4.6441i&6.1666 &+& 4.9424i& 6.1705 &+& 4.9436i\\
$\xi_{5,q}$   &6.4060 &+& 0.0000i&11.0587 &-& 4.6441i &6.1666 &-& 4.9424i& 6.1705 &-& 4.9436i\\
$\xi_{6,q}$  &9.4000 &+& 0.0000i&15.0528 &+& 1.6664i&9.2199 &+& 4.0919i& 9.2242 &+& 4.0959i\\
$\xi_{7,q}$  &9.5787 &+& 3.8753i&15.0528 &-& 1.6664i&9.2199 &-& 4.0919i&9.2242 &-& 4.0959i\\
$\xi_{8,q}$   &9.5787 &-& 3.8753i&38.8187 &+& 0.0000i&11.2067 &+& 1.5205i & 11.2121 &+& 1.5201i  \\
$\xi_{9,q}$  &12.1266 &+& 0.0000i&67.0421 &+& 0.0000i&11.2067 &-& 1.5205i& 11.2121 &-& 1.5201i\\
  \hline
\end{tabular} \\ 
$m = 9$
\end{center}

\underline{(c) $\lambda = 10$ and $\sigma = 1$:} \\ \\
\begin{center}
\begin{tabular}{|c|rcr|rcr|rcr|rcr|}
\hline
 & \multicolumn{6}{c|}{(i) Normal} & \multicolumn{6}{c|}{(ii) Weibull}\\
\hline 
 &    \multicolumn{3}{c|}{Regular} &  \multicolumn{3}{c|}{Coxian}&  \multicolumn{3}{c|}{Regular} &  \multicolumn{3}{c|}{Coxian}  \\
\hline
$\Phi(q)$   &  \multicolumn{3}{c|}{0.7581} &\multicolumn{3}{c|}{0.7581}&\multicolumn{3}{c|}{0.9783}& \multicolumn{3}{c|}{0.9784}\\
\hline
$\xi_{1,q}$   &0.0161 &+& 0.0000i&0.0161 &+& 0.0000i&  0.0127 &+& 0.0000i& 0.0127 &+& 0.0000i\\
$\xi_{2,q}$ &  3.6743 &+& 1.3946i &3.6769 &+& 1.3988i& 4.6032 &+& 2.0864i&4.6035 &+& 2.0868i\\
$\xi_{3,q}$  &  3.6743 &-& 1.3946i& 3.6769 &-& 1.3988i&  4.6032 &-& 2.0864i& 4.6035 &-& 2.0868i\\
$\xi_{4,q}$ & 11.8462 &+& 0.0000i&11.8471 &+& 0.0000i&11.7962 &+& 0.0000i& 11.7963 &+& 0.0000i\\
  \hline
\end{tabular} \\ 
$m =3$ \\ \vspace{0.2cm}
\begin{tabular}{|c|rcr|rcr|rcr|rcr|}
\hline
 & \multicolumn{6}{c|}{(i) Normal} & \multicolumn{6}{c|}{(ii) Weibull}\\
\hline 
 &    \multicolumn{3}{c|}{Regular} &  \multicolumn{3}{c|}{Coxian}&  \multicolumn{3}{c|}{Regular} &  \multicolumn{3}{c|}{Coxian}  \\
\hline
$\Phi(q)$   &  \multicolumn{3}{c|}{0.7626 } &\multicolumn{3}{c|}{0.7627}&\multicolumn{3}{c|}{1.0013}& \multicolumn{3}{c|}{1.0014}\\
\hline
$\xi_{1,q}$   & 0.0161 &+ &0.0000i& 0.0161 &+& 0.0000i&0.0127 &+& 0.0000i& 0.0127 &+& 0.0000i\\
$\xi_{2,q}$   & 3.2677 &+& 2.3936i&3.2316 &+& 2.4245i&3.6033 &+& 3.7915i& 3.5886 &+& 3.8007i\\
$\xi_{3,q}$  & 3.2677 &-& 2.3936i & 3.2316 &-& 2.4245i& 3.6033 &-& 3.7915i & 3.5886 &-& 3.8007i\\
$\xi_{4,q}$   & 4.7176 &+& 0.0000i&5.8504 &+& 1.7320i&6.3675 &+& 0.0000i&  7.2957 &+& 4.2056i\\
$\xi_{5,q}$   &5.0104 &+& 0.0362i &5.8504 &-& 1.7320i&7.2883 &+& 0.0000i& 7.2957 &-& 4.2056i\\
$\xi_{6,q}$  &5.0104 &-& 0.0362i&  7.5588 &+& 0.0000i & 7.3987 &+& 4.0582i &9.8984 &+& 2.1765i\\
$\xi_{7,q}$   & 5.6504 &+& 1.2766i & 11.6475 &+& 0.0000i&7.3987 &-& 4.0582i& 9.8984 &-& 2.1765i\\
$\xi_{8,q}$  & 5.6504 &-& 1.2766i& 11.9483 &+& 0.0000i&9.3094 &+& 1.1112i& 12.1683 &+& 0.0000i\\
$\xi_{9,q}$  & 6.7481 &+& 0.0000i& 18.1680 &+& 0.0000i& 9.3094 &-& 1.1112i&13.4518 &+& 0.0000i\\
$\xi_{10,q}$  & 11.6760 &+& 0.0000i& 23.5959 &+& 0.0000i&11.7366 &+& 0.0000i& 24.4108 &+& 0.0000i\\
  \hline
\end{tabular} \\ 
$m = 9$ \\ \vspace{0.2cm}
\begin{tabular}{|c|rcr|rcr|rcr|rcr|}
\hline
 & \multicolumn{6}{c|}{(iii) Lognormal} & \multicolumn{6}{c|}{(iv) Uniform}\\
\hline 
 &    \multicolumn{3}{c|}{Regular} &  \multicolumn{3}{c|}{Coxian}&  \multicolumn{3}{c|}{Regular} &  \multicolumn{3}{c|}{Coxian}  \\
\hline
$\Phi(q)$   &  \multicolumn{3}{c|}{1.2194} &\multicolumn{3}{c|}{1.2194}&\multicolumn{3}{c|}{1.0386}& \multicolumn{3}{c|}{1.0386}\\
\hline
$\xi_{1,q}$  &  0.0078 &+& 0.0000i& 0.0078 &+& 0.0000i&0.0098 &+& 0.0000i&0.0098 &+& 0.0000i\\
$\xi_{2,q}$  &3.7010 &+& 1.6981i& 3.7010 &+& 1.6981i & 3.5636 &+& 1.6005i & 3.5652 &+& 1.6019i\\
$\xi_{3,q}$ &3.7010 &-& 1.6981i&3.7010 &-& 1.6981i& 3.5636 &-& 1.6005i&3.5652 &-& 1.6019i\\
$\xi_{4,q}$   &11.7522 &+& 0.0000i&11.7522 &+& 0.0000i& 11.8235 &+& 0.0000i&11.8240 &+& 0.0000i\\
  \hline
\end{tabular} \\ 
$m =3$ \\ \vspace{0.2cm}
\begin{tabular}{|c|rcr|rcr|rcr|rcr|}
\hline
 & \multicolumn{6}{c|}{(iii) Lognormal} & \multicolumn{6}{c|}{(iv) Uniform}\\
\hline 
 &    \multicolumn{3}{c|}{Regular} &  \multicolumn{3}{c|}{Coxian}&  \multicolumn{3}{c|}{Regular} &  \multicolumn{3}{c|}{Coxian}  \\
\hline
$\Phi(q)$   &  \multicolumn{3}{c|}{1.2601} &\multicolumn{3}{c|}{1.2609}&\multicolumn{3}{c|}{1.0678}& \multicolumn{3}{c|}{1.0677}\\
\hline
$\xi_{1,q}$   &0.0078 &+& 0.0000i&0.0078 &+& 0.0000i&0.0098 &+& 0.0000i&   0.0098 &+& 0.0000i\\
$\xi_{2,q}$   &3.5668& +& 0.0000i &4.8256 &+& 4.4164i& 2.4537 &+& 3.6815i & 2.4526 &+& 3.6833i\\
$\xi_{3,q}$  & 3.9970 &+& 4.5161i &4.8256 &-& 4.4164i&2.4537 &-& 3.6815i&  2.4526 &-& 3.6833i\\
$\xi_{4,q}$  &3.9970 &-& 4.5161i& 10.1310 &+& 5.8687i&5.6968 &+& 5.4007i &5.7016 &+& 5.4029i\\
$\xi_{5,q}$   & 6.8666 &+ &0.0000i&10.1310 &-& 5.8687i&5.6968 &-& 5.4007i& 5.7016 &-& 5.4029i\\
$\xi_{6,q}$  & 9.2545 &+& 4.5559i&14.6715 &+& 3.8357i &9.0010 &+& 4.9244i &9.0018 &+& 4.9297i \\
$\xi_{7,q}$  & 9.2545 &-& 4.5559i& 14.6715 &-& 3.8357i&   9.0010 &-& 4.9244i&  9.0018 &-& 4.9297i\\
$\xi_{8,q}$  &  9.4033 &+& 0.0000i&16.6606 &+& 0.0000i &11.9574 &+& 2.8322i & 11.9641 &+& 2.8307i \\
$\xi_{9,q}$  &12.6305 &+& 1.5688i&38.8187 &+& 0.0000i &11.9574 &-& 2.8322i &11.9641 &-& 2.8307i \\
$\xi_{10,q}$  &12.6305 &-& 1.5688i& 67.0421 &+& 0.0000i&  12.8358 &+& 0.0000i&   12.8395 &+& 0.0000i\\
  \hline
\end{tabular} \\ 
$m = 9$
\end{center}

\bibliographystyle{abbrv}
\bibliography{PhaseTypebib}
\end{document}